\documentclass[twoside,11pt]{article}

\usepackage{blindtext}

\usepackage{lastpage}

\usepackage{amsmath,amsfonts,amssymb,amsthm}
\usepackage[utf8]{inputenc}
\usepackage{hyperref}       
\usepackage{url}            
\usepackage{booktabs}       
\usepackage{nicefrac}       
\usepackage{microtype}      
\usepackage{xcolor}         
\usepackage{caption}
\usepackage{subcaption}     

\newtheorem{theorem}{Theorem}
\newtheorem{lemma}[theorem]{Lemma} 
\newtheorem{proposition}[theorem]{Proposition}

\newtheorem{definition}[theorem]{Definition}

\newcommand{\bigO}{\mathcal{O}}
\usepackage[nottoc,notlof,notlot]{tocbibind} 
\makeatletter
\let\l@section\l@chapter
\makeatother
\usepackage{amssymb}
\usepackage{amsthm}
\usepackage{hyperref} 
\usepackage{algorithm}
\usepackage{algorithmic}
\usepackage{url}
\usepackage{graphicx}
\usepackage{tcolorbox}
\usepackage{amsthm}
\usepackage{xcolor}
\usepackage{float}
\floatstyle{boxed} 
\usepackage{thm-restate}
\restylefloat{figure}
\usepackage{upgreek}
\usepackage{authblk}
\DeclareMathOperator{\Tr}{Tr}
\DeclareMathOperator{\myspan}{span}
\DeclareMathOperator{\Gr}{Gr}

\DeclareMathOperator{\diag}{diag}
\DeclareMathOperator{\dist}{dist}
\DeclareMathOperator{\grad}{grad}
\DeclareMathOperator{\Hess}{Hess}
\DeclareMathOperator{\Exp}{Exp}
\DeclareMathOperator{\Log}{Log}
\DeclareMathOperator{\atan}{atan}

\begin{document}

\title{A Nesterov-style Accelerated Gradient Descent Algorithm for the Symmetric Eigenvalue Problem}

\author{Foivos Alimisis, Simon Vary and Bart Vandereycken}

\maketitle

\begin{abstract}
We develop an accelerated gradient descent algorithm on the Grassmann manifold to compute the subspace spanned by a number of leading eigenvectors of a symmetric positive semi-definite matrix. This has a constant cost per iteration and a provable iteration complexity of $\Tilde \bigO(1/\sqrt{\delta})$, where $\delta$ is the spectral gap and $\Tilde \bigO$ hides logarithmic factors. This improves over the $\Tilde \bigO(1/\delta)$ complexity achieved by subspace iteration and standard gradient descent, in cases that the spectral gap is tiny. It also matches the iteration complexity of the Lanczos method that has however a growing cost per iteration. On the theoretical part, we rely on the formulation of Riemannian accelerated gradient descent by \cite{zhang2018towards} and new characterizations of the geodesic convexity of the symmetric eigenvalue problem by \cite{alimisis2022geodesic}. On the empirical part, we test our algorithm in synthetic and real matrices and compare with other popular methods.

\end{abstract}





\section{Introduction}

The algebraic eigenvalue problem is one of the original research topics in numerical linear algebra \cite{golub2013matrix,hornMatrixAnalysis2012a,saadNumericalMethodsLarge2011}, with numerous applications in statistics and engineering that remain relevant today. One of the most famous applications is principal component analysis (PCA), as it turns out that the principal components of a dataset correspond to eigenvectors of its covariance matrix. The importance of statistical processes like PCA in analyzing data has made eigenvalue problems a core topic of discussion among the modern machine learning community as well \cite{alimisis2021distributed,hardt2014noisy,pmlr-v119-huang20e}.

We adopt the following simple statement of the problem:
\begin{tcolorbox}
Consider a symmetric positive semi-definite matrix $A \in \mathbb{R}^{n \times n}$ and compute a subspace spanned by the $p$ leading eigenvectors of $A$, that is, the eigenvectors associated with the $p$ largest eigenvalues. Here, $1 \leq p \leq n$. 
\end{tcolorbox}
Without loss of generality, we sort the $n$ eigenvalues $\lambda_1, \ldots, \lambda_n$ of $A$ (counted with multiplicity) in descending order: $\lambda_1 \geq \lambda_2 \geq ... \geq \lambda_n \geq 0$. We denote by $\delta := \lambda_p-\lambda_{p+1}
$ the spectral gap between the $p$th and $(p+1)$th eigenvalue. Remark that the assumption on the positive semi-definiteness of the matrix $A$ can be done without loss of generality, since we can shift $A$ like $A+\alpha I_n$ with $I_n$ the $n \times n$ identity matrix and $\alpha$ a sufficiently large scalar to make $A+\alpha I$ positive semi-definite. This operation does not change the eigenvectors or the spectral gap. As we will see later from the structure of the algorithms and the convergence results, (accelerated) gradient methods are not affected by this transformation. This good property is also shared by the Lanczos iteration but not by subspace iteration. The former method has however not a constant cost per iteration and needs to be restarted after many iterations.

For convenience, let us denote $\Lambda_{\alpha}=\textnormal{
diag}(\lambda_1,...,\lambda_p)$ the diagonal matrix featuring the wanted eigenvalues and $\Lambda_{\beta}=\textnormal{
diag}(\lambda_{p+1},...,\lambda_n)$ the diagonal matrix featuring the unwanted ones. Similarly, let $V_{\alpha}=[v_1,...,v_p]$ and $V_{\beta}=[v_{p+1},...,v_n]$ be orthonormal matrices featuring the desired and undesired eigenvectors of $A$,  respectively. Note that $V_{\beta}$ is the orthogonal complement of $V_{\alpha}$ and concatenating the two matrices, we get the orthogonal matrix $V=[V_{\alpha} \hspace{1mm} V_{\beta}]$ having all the eigenvectors of $A$ in its columns.

The well-known Ky--Fan theorem states that one can compute the $p$ leading eigenvectors of $A$ by minimizing\footnote{It might seem more natural to maximize $\Tr(X^T A X)$ but later one we will use properties from convex optimization, where it is more natural to minimize functions.} the function
\begin{equation*}
    f(X)=-\Tr(X^T A X)
\end{equation*}
over the set of all orthonormal matrices $X \in \mathbb{R}^{n \times p}$. The minimum of this function is $-(\lambda_1+...+\lambda_p)=-\Tr(\Lambda_{\alpha})$. 

It follows from this formulation of the problem that a convenient space to solve it is the Stiefel manifold:
\begin{equation*}
    \textnormal{St}(n,p)=\lbrace X \in \mathbb{R}^{n \times p} \colon X^T X=I_k \rbrace.
\end{equation*}
Even though this strategy seems appealing, the Stiefel manifold is a difficult space to work with, because the formula for the exponential map (geodesics) is quite involved and the inverse of the exponential map (logarithm) is not known in closed form. In addition, minimizers of $f$ are not unique on $\textnormal{St}(n,p)$ even with strictly non-zero spectral gap.

A viable alternative comes into play by noticing that the function $f$ is invariant under orthogonal tranformations of the matrix $X$. This implies that the function
\begin{equation*}
    f(\textnormal{span}(X))=-\Tr(X^T A X)
\end{equation*}
is well-defined over the set of $p$ dimensional subspaces of $\mathbb{R}^n$ represented by an orthonormal matrix. This space admits a much more convenient Riemannian structure and is called the \emph{Grassmann} manifold. The Grassmann manifold is very well studied among the Riemannian optimization community and we refer the reader to \cite{absilOptimizationAlgorithmsMatrix2008},  \cite{alimisis2022geodesic} and \cite{bendokatGrassmannManifoldHandbook2020}, for a complete presentation of the basic notions. In Section \ref{sec:geom_grassmann}, we also present a short summary of Riemannian geometry concepts used in this paper, as well as the special Grassmann-related calculations. Its performance depends on the application, but there are cases where we can confidently argue that this can be the favourite algorithm for eigenvalue-eigenvector computation.

\section{Contributions and related work}

\subsection{Contributions}

Our contribution is the theoretical and experimental analysis of a version of Nesterov's accelerated gradient descent \cite{nesterov1983method} on the Grassmann manifold for calculating a subspace spanned by the $p$ leading eigenvectors of a matrix $A$. To that end, we rely on the rich literature of general Riemannian algorithms, and more specifically on the formulation of Riemannian accelerated gradient descent by \cite{zhang2018towards}. The other part of our analysis relies on the geodesic convexity characterization of the function $f$ (the so-called block Rayleigh quotient) on the Grassmann manifold done by \cite{alimisis2022geodesic}. The latter proves the important property that $f$ is \emph{geodesically weakly-strongly-convex} (Theorem 8 in \cite{alimisis2022geodesic}). Despite that the estimate sequences technique of \cite{zhang2018towards} targets only geodesically \emph{strongly convex} objectives, there is already a technique to design estimate sequences for a weakly-strongly-convex function in the Euclidean regime due to \cite{bu2020note}. Thus, from a technical standpoint, we need to merge the Riemannian approach for strongly convex and the Euclidean approach for weakly-strongly-convex functions. To that end, the geodesic search technique for selecting the momentum coefficient analyzed in \cite{alimisis2021momentum} (which extended the similar Euclidean technique of \cite{nesterov2020primal}) will be of great help. This approach yields provable accelerated convergence guarantees for our algorithm. On the experimental side, we show that our algorithm is competitive compared to other state-of-the-art eigensolvers.

\subsection{Accelerated optimization on manifolds}

Motivated by the classic work of Nesterov \cite{nesterov1983method}, a plethora of works focusing on accelerated methods on Riemannian manifolds has been developed in recent years. We refer the reader to \cite{ahn2020nesterov,zhang2018towards} for algorithms targeting geodesically strongly convex objective functions. There is a second line of work targeting objectives that are geodesically convex but not strongly convex with moderate success so far \cite{alimisis2021momentum,pmlr-v162-kim22k}. In both cases, the main obstacles consist of designing an estimate sequence that can handle the non-linearity of the manifold. There is also a recent line of work on no-go results on acceleration on manifolds, and namely that one cannot hope for any \emph{global} accelerated method on a manifold of negative sectional curvatures \cite{criscitiello2022negative,hamilton2021no}. The latter are not directly applicable in our case, since we work with the Grassmann manifold, which is of nonnegative sectional curvatures. However, they highlight the difficulties of designing accelerated methods on manifolds, and they give indications of why this could be achieved only locally.

\subsection{Accelerated methods for the symmetric eigenvalue problem}

The simplest method for computing eigenvectors and eigenvalues of a symmetric matrix is the subspace iteration. However, it turns out that this method is quite slow, both theoretically ($\Tilde{\bigO}(1/\delta)$ iteration complexity) and practically.
A significant part of research in numerical linear algebra has to do with ``accelerating" 
vanilla methods like subspace iteration using more complicated mechanisms. The most well-known accelerated scheme for computing leading eigenvectors is the Lanczos method, which is a member of the family of Krylov methods. The Lanczos method has iteration complexity of $\Tilde \bigO (1/\sqrt{\delta})$ and improves over subspace iteration. This method is, however, not stationary  since it enlarges an approximation subspace in every step (like any Krylov method). The cost per iteration therefore grows both in time and in memory. This iteration is therefore restarted in practice. While the restarting strategy is empirically effective, it makes the method more complicated to use and analyze. In this paper, we therefore focus on methods that are accelerated versions of stationary methods, like steepest descent. They have the benefit of a constant cost per iteration.

An example of accelerating subspace iteration has been done employing the technology of Polyak's momentum (heavy ball) method by \cite{xu2018accelerated}. The resulting deterministic scheme of this paper is a subspace iteration with an extra momentum term that has guaranteed convergence in at most $\Tilde \bigO (1/\sqrt{\delta})$ many iterations, if the momentum coefficient is chosen precisely in terms of $\lambda_{p+1}$ (the $(p+1)$th largest eigenvalue). If $\lambda_p$ and $\lambda_{p+1}$ are not known in advance (which is usually the case), then the convergence behaviour of this algorithm can worsen considerably. 

The algorithm of \cite{xu2018accelerated} is essentially a modern reformulation of the classical Chebyshev iteration (see \cite{saadNumericalMethodsLarge2011}). An interesting contribution in \cite{xu2018accelerated} (except from the main contribution, a stochastic version of the algorithm) is a clever way to implement their algorithm (essentially Chebyshev iteration) in a numerically stable manner (Lemma 12), paying the extra cost of a QR-decomposition in a $(2n) \times p$ matrix (instead of $n \times p$). A different approach based on non-linear conjugate gradients is presented in \cite{alimisis2023gradient}. The conjugate gradient method combined with a choice of step size via an exact line search has excellent empirical performance, but it is very hard to prove any theoretical convergence guarantees (\cite{alimisis2023gradient} does not provide any). Other interesting methods that are empirically accelerated but come without much theory are LOBPCG (\cite{knyazev2001toward}) and \cite{absil2009accelerated}.

\subsection{Riemannian accelerated gradient descent on Grassmann manifold}

In this paper we deviate from the previous research directions and develop a version of Nesterov's accelerated gradient descent on the Grassmann manifold for the symmetric eigenvalue problem. When measured in terms of matrix-vector products, every iterate of this algorithm has double the cost as subspace iteration and subspace iteration with momentum \cite{xu2018accelerated}. The algorithm does in addition incur overheads when computing the momentum terms and the geodesic. These costs are however not dependent on $A$ and involve only dense linear algebra routines that are typically very optimized in practical implementations.

The analysis of our method reveals that one needs at most $\Tilde{\bigO}(1/\sqrt{\delta})$ many iterations to compute the dominant subspace with accuracy $\epsilon$, if the initialization is $\bigO(\delta^{3/4})$ close to the optimal subspace. The need for local initialization is an artifact of the general analysis of the Riemannian version of accelerated gradient descent we use, developed in \cite{zhang2018towards}. Also, our algorithm relies on an almost exact knowledge for the gap $\delta=\lambda_p-\lambda_{p+1}$, similarly to \cite{xu2018accelerated} which requires exact knowledge of $\lambda_p$ and $\lambda_{p+1}$.


\section{Geometry of the Grassmann manifold and calculations for the block Rayleigh quotient}
\label{sec:geom_grassmann}

We present here a brief introduction into the geometry of the Grassmann manifold. The content is not new and for more details, we refer to \cite{absilOptimizationAlgorithmsMatrix2008,bendokatGrassmannManifoldHandbook2020,edelmanGeometryAlgorithmsOrthogonality1999}.

The $(n,p)$-Grassmann manifold is defined as the set of all $p$ dimensional subspaces of $\mathbb{R}^n$:
\begin{equation*}
    \Gr(n,p)=\lbrace \mathcal{X} \subseteq \mathbb{R}^n \colon \mathcal{X} \hspace{1mm}  \text{is a subspace and} \dim(\mathcal{X})=p \rbrace.
\end{equation*}

Any element $\mathcal{X}$ of $\Gr(n,p)$ can be represented by a matrix $X \in \mathbb{R}^{n \times p}$ that satisfies $\mathcal{X} = \myspan(X)$. Such a representative is not unique since $Y=XQ$ for some invertible matrix $Q \in  \mathbb{R}^{p \times p}$ satisfies $\myspan(Y) = \myspan(X)$. Without loss of generality, we will therefore always take matrix representatives $X$ of subspaces $\mathcal{X}$ that have orthonormal columns throughout the paper. 
With some care, the non-uniqueness of the representatives is not a problem.\footnote{This can be made very precise by describing $\Gr(n,p)$ as the quotient of the Stiefel manifold with the orthogonal group. The elegant theory of this quotient manifold is worked out in \cite{absilOptimizationAlgorithmsMatrix2008}.} As mentioned in the introduction, our objective function $f$ is invariant to $Q$.


\paragraph{Riemannian structure.} The set $\Gr(n,p)$ admits the structure of a differential manifold with tangent spaces
\begin{equation}\label{eq:def_TXGr}
    T_{\mathcal{X}} \Gr(n,p)=\lbrace G \in \mathbb{R}^{n \times p} \colon X^T G=0 \rbrace, 
\end{equation}
where $\mathcal{X} = \myspan(X)$. Since $X^T G = 0$ if and only if $(XQ)^T G=0$, for any invertible matrix $Q \in \mathbb{R}^{p \times p}$, this description of the tangent space does not depend on the representative $X$. However, a specific tangent vector $G$ will depend on the chosen $X$. With slight abuse of notation,\footnote{Using the quotient manifold theory, one would use horizontal lifts.} the above definition should therefore be interpreted as: given a fixed $X$, we define tangent vectors $G_1, G_2, \ldots $ of  $\Gr(n,p)$ at $\mathcal{X}=\myspan(\mathcal{X})$. 

This subtlety is important, for example, when defining an inner product on $T_{\mathcal{X}} \Gr(n,p)$:
\[
 \langle G_1, G_2 \rangle_{\mathcal{X}} = \Tr(G^T_1 G_2) \ \text{\ with\  $G_1,G_2 \in T_{\mathcal{X}} \Gr(n,p)$ }.
\] 
Here, $G_1$ and $G_2$ are tangent vectors of the same representative $X$. Observe that the inner product is invariant to the choice of orthonormal representative: If $\Bar{G}_1=G_1 Q$ and $\Bar{G}_2 = G_2 Q$ with orthogonal $Q$, then we have
\begin{equation*}
    \langle \bar G_1, \bar G_2 \rangle_{\mathcal{X}} = \Tr(\bar G^T_1  \bar G_2) = \Tr(Q^T G_1^T G_2 Q)= \Tr(G_1^T G_2 Q Q^T) = \Tr(G_1^T G_2).
\end{equation*}
It is easy to see that the norm induced by this inner product in any tangent space is the Frobenius norm, which we will denote throughout the paper as $\| \cdot \|:=\| \cdot \|_F$.

\paragraph{Exponential map/Retractions.} Given the Riemannian structure of $\Gr(n,p)$, we can compute the exponential map at a point $\mathcal{X}$ as \cite[Thm.~3.6]{absilRiemannianGeometryGrassmann2004}
\begin{equation}\label{eq:formula_geo}
\begin{aligned}
 \Exp_{\mathcal{X}}: T_{\mathcal{X}} \Gr(n,p) &\rightarrow \Gr(n,p) \\ 
  G &\mapsto \myspan(\, X V \cos(\Sigma) + U \sin(\Sigma) \, ), 
\end{aligned}
\end{equation}
where $ U \Sigma V^T$ is the \emph{compact} SVD of $G$ such that $\Sigma$ and $V$ are square matrices. 

The exponential map is invertible in the domain \cite[Prop.~5.1]{bendokatGrassmannManifoldHandbook2020}
\begin{equation}\label{eq:inj_exp}
\left \lbrace G \in T_{\mathcal{X}} \Gr(n,p)  \colon  \| G \|_2 < \frac{\pi}{2} \right \rbrace,
\end{equation}
where $\| G \|_2$ is the spectral norm of $G$. The inverse of the exponential map restricted to this domain is the logarithmic map, denoted by $\Log$. Given two subspaces $\mathcal{X},\mathcal{Y}\in \Gr(n,p)$, we have
\begin{equation}\label{eq:log formula}
 \Log_{\mathcal{X}}(\mathcal{Y}) = U \atan(\widehat\Sigma) \, V^T,
\end{equation}
where  $U \widehat \Sigma V^T = (I - X X^T) Y (X^T Y)^{-1}$ is again a compact SVD. This is well-defined if $X^T Y$ is invertible, which is guaranteed if all principal angles between $\mathcal{X}$ and $\mathcal{Y}$ are strictly less than $\pi / 2$ (see below). By taking $G = \Log_{\mathcal{X}}(\mathcal{Y})$, we see that $\Sigma = \atan(\widehat\Sigma)$.

Given that the computation of the exponential map can become costly, since there are many $\sin()$ and $\cos()$ quantities involved, there is an obvious interest in approximations. A first order and smooth approximation of the exponential map is called a \emph{retraction}. A valid choice for a retraction on the Grassmann manifold is
\begin{equation}\label{eq:def_retr_qr}
    \textnormal{Retr}_{\mathcal{X}}(G)=qf(X+G),
\end{equation}
where $qf(X+G)$ is the $Q$ factor of the QR decomposition of $X+G$. This is a first order approximation of the exponential map in the sense that the curves $t \rightarrow \Exp_{\mathcal{X}}(tG)$ and $t \rightarrow \textnormal{Retr}_{\mathcal{X}}(tG)$, starting from $\mathcal{X}$, have the same first derivative at $t=0$. The use of retractions in Riemannian optimization seems appealing due to their simplicity and low computational costs, but note that a retraction does not depend at all on the Riemannian metric. Thus, if one wishes to prove non-asymptotic convergence guarantees for algorithms on Riemannian manifolds, the exponential map exposes a much richer Riemannian structure and provides much stronger geometric bounds (like Lemma \ref{le:geom_bound}).

\paragraph{Principal angles.} The Riemannian structure of the Grassmann manifold can be conveniently described by the notion of the principal angles between subspaces. Given two subspaces $\mathcal{X},\mathcal{Y} \in \Gr(n,p)$, the principal angles between them are $0 \leq \theta_1 \leq \cdots \leq \theta_p \leq \pi/2$ obtained from the SVD 
\begin{equation}\label{eq:SVD_for_principal_angles}
    Y^T X=U_1 \cos \theta \ V_1^T
\end{equation}
where $U_1 \in \mathbb{R}^{p \times p}, V_1 \in \mathbb{R}^{p \times p}$ are orthogonal and the diagonal matrix $\cos \theta= \diag(\cos \theta_1,...,\cos \theta_p)$.

We can express the Riemannian logarithm using principal angles and the intrinsic distance induced by the Riemannian inner product discussed above is
\begin{equation}\label{eq:distance_with_Log_and_angles}
    \dist(\mathcal{X},\mathcal{Y)}=\| \Log_{\mathcal{X}} (\mathcal{Y}) \| = \| \Log_{\mathcal{Y}} (\mathcal{X}) \|=\sqrt{\theta_1^2+...+\theta_p^2}=\| \theta \|_2,
\end{equation}
where $\theta=(\theta_1, \ldots ,\theta_p)^T$.

If $X \in \mathbb{R}^{n \times p}$ is an arbitrary matrix with orthonormal columns, then, generically, these columns will not be exactly orthogonal to the $p$ leading eigenvectors $v_1, \ldots, v_p$ of $A$. Thus, we have with probability one that the principal angles between $\mathcal{X}$ and the space of $k$ leading eigenvectors satisfy $0 \leq \theta_1 \leq \cdots \leq \theta_p < \pi/2$.

\paragraph{Curvature.}
We can compute exactly the sectional curvatures in $\Gr(n,p)$, but for our purposes we only need an upper and lower bound. A lower bound is $0$, while an upper bound is $2$, see \cite{Wong,bendokatGrassmannManifoldHandbook2020}. These curvature bounds give rise to many geometric bounds on the Grassmann manifold. For our analysis, we use the following two. The first is a geometric bound proved in \cite{zhang2018towards} and the second is a classical result which follows by Toponogov's theorem:

\begin{lemma}
\label{le:geom_bound}
    Let $a,b,c,d$ be four points in a geodesically uniquely convex subset of a Riemannian manifold, with sectional curvatures in the interval $[-K,K]$ and
    \begin{equation*}
       \max \lbrace \textnormal{dist}(c,a), \textnormal{dist}(d,a) \rbrace \leq \frac{1}{4 \sqrt{K}},
    \end{equation*}
    then
    \begin{equation*}
    \| \textnormal{Log}_{d}(a) - \textnormal{Log}_{d}(b) \|^2 \leq (1+5K\max \lbrace \textnormal{dist}(c,a), \textnormal{dist}(d,a) \rbrace^2) \| \textnormal{Log}_{c}(a) - \textnormal{Log}_{c}(b) \|^2.    \end{equation*}
\end{lemma}
By the previous discussion, if $a,b,c,d$ are subspaces on the Grassmann manifold with the previously analysed Riemannian structure, we can take $K=2$.

\begin{lemma}
\label{prop:tangent_space}
Let $a,b,c$ three points on a manifold of nonnegative sectional curvatures, such that the geodesics between them exist and are unique. 
 Then
\begin{equation*}
    \textnormal{dist}(a , b) \leq \| \Log_c(a)-\Log_c(b) \|.
\end{equation*}

\end{lemma}

In our case, it is true that the Grassmann manifold is of nonnegative sectional curvatures and a sufficient condition for the geodesics among two subspaces $\mathcal{X}$ and $\mathcal{Y}$ to be unique is
\begin{equation*}
    \textnormal{dist}(\mathcal{X},\mathcal{Y}) < \frac{\pi}{2}.
\end{equation*}

\paragraph{Block Rayleigh quotient.}
Our objective function for minimization is the block version of the Rayleigh quotient:
\begin{equation}\label{eq:block_f_on_Grassmann}
f(\mathcal{X}) = -\Tr(X^TAX) \text{ where $\mathcal{X} = \myspan(X) \in \Gr(n,p)$ s.t. $X^T X = I_p$}.
\end{equation}
This function has $\mathcal{V}_{\alpha}=\myspan(\begin{bmatrix} v_1 \hspace{2mm} \cdots \hspace{2mm} v_p \end{bmatrix})$ as global minimizer. This minimizer is unique on $\Gr(n,p)$ if and only if $\delta :=\lambda_p-\lambda_{p+1}>0$. The minimum of $f$ is $f(\mathcal{V}_{\alpha})$ and we denote it by $f^*$.

Given any differentiable function $f\colon\Gr(n,p) \rightarrow \mathbb{R}$, we can define its Riemannian gradient as the vector field that satisfies
\begin{equation*}
    df(\mathcal{X})(G) = \langle \grad f(\mathcal{X}),G \rangle_{\mathcal{X}},
\end{equation*}
for all $G \in T_{\mathcal{X}} \textnormal{Gr}(n,p)$.

For a given representative $X$ of $\mathcal{X}$, the Riemannian gradient of the block Rayleigh quotient satisfies 
\begin{equation*}\label{eq:grad f formula}
 \grad f(\mathcal{X}) = -2(I-X X^T) A X.
\end{equation*}
Using the notions of the Riemannian gradient and Levi-Civita connection, we can define also a Riemannian notion of Hessian. 
For the block Rayleigh quotient $f$, the Riemannian Hessian $\Hess f$ evaluated as bilinear form satisfies
\begin{equation}\label{eq:Hessian_f_inner_product}
  \Hess f(\mathcal{X})[{G},{G}] = 2 \langle G, G X^T A X - A G \rangle, 
\end{equation}
for  $G \in T_{\mathcal{X}} \Gr(n,p)$; see \cite[\S4.4]{edelmanGeometryAlgorithmsOrthogonality1999} or \cite[\S6.4.2]{absilOptimizationAlgorithmsMatrix2008}.

As mentioned in the introduction, the paper \cite{alimisis2022geodesic} provides several properties of the block Rayleigh quotient $f$, that are useful for our analysis:
\begin{proposition}[Smoothness]\label{prop:smoothness}
The eigenvalues of the Riemannian Hessian of $f$ on $\Gr(n,p)$ are upper bounded by  $\gamma := 2 (\lambda_1 - \lambda_n)$. 
\end{proposition}
As it is folklore, such a bound on the eigenvalues of the Riemannian Hessian can be used to derive many other inequalities, including the following:
\begin{equation*}\label{eq:quadratic_upper_bound}
    f(\mathcal{X}) \leq f(\mathcal{Y})+ \langle \grad f(\mathcal{Y}), \Log_{\mathcal{Y}} (\mathcal{X}) \rangle+\frac{\gamma}{2} \dist^2(\mathcal{X},\mathcal{Y}), 
\end{equation*}
for any $\mathcal{X}, \mathcal{Y} \in \Gr(n,p)$ such that $\Log_{\mathcal{X}}(\mathcal{Y})$ is well-defined. The last inequality gives a useful bound for the case that $\mathcal{X}$ is obtained from $\mathcal{Y}$ via a gradient step:
\begin{proposition}
    If $\mathcal{X}=\Exp_\mathcal{Y}(-\eta \, \textnormal{grad} f(\mathcal{Y}))$, $\eta>0$, then
    \begin{equation*}
        f(\mathcal{X}) \leq f(\mathcal{Y}) - \left(\eta-\frac{\gamma\eta^2}{2} \right) \| \textnormal{grad}f(\mathcal{Y}) \|^2.
    \end{equation*}
\end{proposition}
\label{eq:func_val_dec_exp}
The quantity $\eta-\frac{\gamma\eta^2}{2}$ is maximized for $\eta=\frac{1}{\gamma}$, in which case the bound takes the form 
\begin{equation}
\label{eq:func_val_decrease}
    f(\mathcal{X}) \leq f(\mathcal{Y}) - \frac{1}{2\gamma} \| \textnormal{grad}f(\mathcal{Y}) \|^2.
\end{equation}

Such a function value decrease is a mild property of descent algorithms that can be satisfied even when the gradient step is taken using the retraction $\textnormal{Retr}_\mathcal{X}$ from~\eqref{eq:def_retr_qr} and a step-size obtained via exact line search. Remarkably, such a line search can be implemented very efficiently, as shown in \cite{alimisis2023gradient}.

\begin{proposition}[Lemma 3 in \cite{alimisis2023gradient}]
\label{prop:func_val_decrease_retr}
    If $\mathcal{X}=\textnormal{Retr}_\mathcal{Y}(-\eta_{\text{optimal}} \cdot \textnormal{grad} f(\mathcal{Y}))$, where $\eta_{\text{optimal}}$ is obtained via an exact line search, then
    \begin{equation*}
        f(\mathcal{X}) \leq f(\mathcal{Y}) - \frac{2}{5 \gamma} \| \textnormal{grad}f(\mathcal{Y}) \|^2.
    \end{equation*}
\end{proposition}

We now examine some deeper facts about the block Rayleigh quotient, which are analysed in detail in \cite{alimisis2022geodesic}. We begin with the quadratic growth condition (Proposition 5 in \cite{alimisis2022geodesic}):
\begin{proposition}
\label{prop:quadratic_growth}
Let $0 \leq \theta_1 \leq \cdots \leq \theta_k < \pi/2$ be the principal angles between the subspaces $\mathcal{X}$ and $\mathcal{V}_\alpha$. The function $f$ satisfies
$$f(\mathcal{X})-f^* \geq c_Q \, \delta \, \dist^2(\mathcal{X},\mathcal{V_{\alpha}})$$
where $c_Q = 4/\pi^2 > 0.4$.
\end{proposition}

A central property of the block Rayleigh quotient is the so-called geodesic weak-strong convexity (Theorem 8 in \cite{alimisis2022geodesic}). It reads as follows:
\begin{proposition}[Weak-strong convexity]\label{thm:weak_strong_convex}
Let $0 \leq \theta_1 \leq \cdots \leq \theta_p < \pi/2$ be the principal angles between the subspaces $\mathcal{X}$ and $\mathcal{V}_\alpha$. Then, $f$ satisfies
\begin{equation*}
   f(\mathcal{X})-f^* \leq \frac{1}{a(\mathcal{X})} \langle \grad f(\mathcal{X}), -\Log_{\mathcal{X}}(\mathcal{V}_{\alpha}) \rangle - c_Q \delta \, \dist^2 (\mathcal{X},\mathcal{V}_{\alpha})
\end{equation*}
with $a(\mathcal{X})= \theta_p / \tan \theta_p >0$, $c_Q = 4/\pi^2 > 0.4$, $\delta = \lambda_p - \lambda_{p+1} \geq 0$ and $f^*=f(\mathcal{V}_{\alpha})$.
\end{proposition}

The term weak-strong convexity comes from the fact that this inequality is similar to the one of strong convexity but it holds only between an arbitrary $\mathcal{X}$ and the optimum $\mathcal{V}_{\alpha}$ (and not between any pair of points). Also, the upper bound is weaker in the sense that it features the constant $1/a(\mathcal{X})$ which is in general larger than $1$.

In the rest of the paper, we denote $\mu:=2 c_Q \delta$ (this quantity plays the role of a strong convexity constant). Then, the inequality of Proposition \ref{thm:weak_strong_convex} takes the form:
\begin{equation}
\label{eq:weak-strong-conv}
    f(\mathcal{X})-f^* \leq \frac{1}{a(\mathcal{X})} \langle \grad f(\mathcal{X}), -\Log_{\mathcal{X}}(\mathcal{V}_{\alpha}) \rangle - \frac{\mu}{2} \dist^2 (\mathcal{X},\mathcal{V}_{\alpha}).
\end{equation}

\section{Weak estimate sequence}

For reasons related to both the weak nature of geodesic convexity and the non-linearity of the working domain (Grassmann manifold), 
we introduce a \emph{weaker notion} of the classical estimate sequence than \cite{nesterov1983method}. This is the strategy in both \cite{bu2020note} and \cite{zhang2018towards}.

\begin{definition}
    A weak estimate sequence for $f$ is a sequence of functions $( \phi_k )_{k=0}^{\infty}$ defined on the Grassmann manifold, and a sequence of positive scalars $( \tau_k)_{k=0}^{\infty}$, such that
    \begin{equation*}
        \lim_{k \rightarrow \infty} \tau_k = 0 \hspace{5mm} \text{and} \hspace{5mm} \phi_k(\mathcal{V_{\alpha}}) \leq (1-\tau_k) f(\mathcal{V_{\alpha}})+\tau_k \phi_0(\mathcal{V_{\alpha}})
    \end{equation*}
    where $\mathcal{V_{\alpha}} = \textnormal{argmin}_{\mathcal{X} \in \Gr(n,p)} f(\mathcal{X})$.
We denote such a weak estimate sequence by the pair $(\tau_k,\phi_k)$.
\end{definition}

The difference with the classical definition is that the inequality holds only at the optimum $\mathcal{V}_{\alpha}$ and not at any point.

We utilize weak estimate sequences in the following way:

\begin{proposition}\label{prop:utility_weak_est_seq}
\label{prop:bound*}
    If for some sequence of subspaces $(\mathcal{X}_k )_{k=0}^{\infty}$, we have
    \begin{equation*}
        f(\mathcal{X}_k) \leq \phi_k^* : = \min_{\mathcal{X} \in \Gr(n,p)} \phi_k(\mathcal{X})
    \end{equation*}
    where $(\phi_k, \tau_k)$ is a weak estimate sequence, then 
    \begin{equation*}
        f(\mathcal{X}_k)-f^* \leq \tau_k (\phi_0(\mathcal{V}_{\alpha})-f^*).
    \end{equation*}
\end{proposition}

\begin{proof}
    The proof is direct by the fact that
    \begin{align*}
        f(\mathcal{X}_k) &\leq \min_{\mathcal{X} \in \Gr(n,p)} \phi_k(\mathcal{X}) \leq \phi_k(\mathcal{V}_{\alpha}) \\
        & \leq (1-\tau_k) f(\mathcal{V}_{\alpha})+\tau_k \phi_0(\mathcal{V}_{\alpha})=(1-\tau_k) f^*+\tau_k \phi_0(\mathcal{V}_{\alpha}).
    \end{align*}
    Rearranging we get the result.        
\end{proof}

Now, we describe how to construct a weak estimate sequence for our geodesically weak-strongly-convex function $f$ defined in~\eqref{eq:block_f_on_Grassmann}. The result below is valid for any function that satisfies~\ref{eq:weak-strong-conv} but we phrase it directly for $f$ for simplicity. 

\begin{proposition} \label{prop:weak_estimate_sequence_general}
    Let $f$ be~\eqref{eq:block_f_on_Grassmann}.
    Choose an arbitrary function $\phi_0 \colon \Gr(n,p) \to \mathbb{R}$ and an arbitrary sequence $(\mathcal{Y}_k)_{k=0}^\infty$ of subspaces in $\Gr(n,p)$. We also choose a sequence $(\alpha_k)_{k=0}^\infty$ of scalars such that $\alpha_k \in (0,1)$ and $\Sigma_{k=0}^{\infty} \alpha_k = \infty$.

    Define $\tau_0=1$. For all $k \geq 0$, let $B_k$ be a lower bound for $a(\mathcal{Y}_k)$ (as defined in Proposition~\ref{thm:weak_strong_convex}) and define
    \begin{align*}
        \tau_{k+1}&:=(1-\alpha_k) \tau_k \\ 
        \bar \phi_{k+1}(\mathcal{X})&:=(1-\alpha_k) \phi_k(\mathcal{X}) \\
          &\qquad +\alpha_k \left(f(\mathcal{Y}_k)+\frac{1}{B_k} \langle \textnormal{grad}f(\mathcal{Y}_k), \textnormal{Log}_{\mathcal{Y}_k}(\mathcal{X}) \rangle + \frac{\mu}{2} \textnormal{dist}^2 (\mathcal{Y}_k,\mathcal{X}) \right)        
    \end{align*}
    If $\phi_{k}(\mathcal{V}_{\alpha}) \leq \bar \phi_{k}(\mathcal{V}_{\alpha})$ for all $k \geq 0$, then the pair $(\tau_k,\phi_k)$ is a weak estimate sequence for $f$.
\end{proposition}

\begin{proof}
   We prove the main inequality involving $\phi_k$ in the definition of a weak estimate sequence by induction. For $k=0$, we have $ \phi_0(\mathcal{V}_{\alpha})=(1-\tau_0) f^* + \tau_0 \phi_0(\mathcal{V}_{\alpha})$ because $\tau_0=1$.  Assume that the inequality holds for some $k \geq 0$:
   $$\phi_k(\mathcal{V}_{\alpha})-f^* \leq \tau_k(\phi_0(\mathcal{V}_{\alpha})-f^*).$$ 
   Then, we have
   \begin{align*}
       \phi_{k+1}(\mathcal{V}_{\alpha})-f^* &\leq \bar \phi_{k+1}(\mathcal{V}_{\alpha}) - f^* \\ &\leq (1-\alpha_k)\phi_k(\mathcal{V}_{\alpha}) + \alpha_k f^* - f^* \\ & = (1-\alpha_k) ( \phi_k(\mathcal{V}_{\alpha})-f^* ) \\ & \leq (1-\alpha_k) \tau_k (\phi_0(\mathcal{V}_{\alpha})-f^*) \\ & = \tau_{k+1} (\phi_0(\mathcal{V}_{\alpha})-f^*).
   \end{align*}
   Thus, the inequality holds also for $k+1$ which concludes the induction. The first inequality follows from the construction of $\phi_k$, the second by inequality \ref{eq:weak-strong-conv} (and $a(\mathcal{Y}_k) \geq B_k$) and the third by the induction hypothesis.

Furthermore, we observe that the assumption $\Sigma_{k=0}^{\infty} \alpha_k = \infty$ guarantees that $\lim_{k \rightarrow \infty} \tau_k = 0$, which finishes the proof. 
\end{proof}

\section{Towards an algorithm}
We now use Proposition \ref{prop:weak_estimate_sequence_general} to construct a more specific weak estimate sequence:

\begin{proposition}
\label{prop:weak_estimate_sequence_special}
Consider $\phi_k$, $\alpha_k$, $B_k$ and $\mathcal{Y}_k$ as in Proposition \ref{prop:weak_estimate_sequence_general} and let $\phi_k^*$ be defined as in Proposition \ref{prop:utility_weak_est_seq}. 
    Choose $\phi_0(\mathcal{X})=\phi_0^*+ \frac{\gamma_0}{2} \| \textnormal{Log}_{\mathcal{Y}_0}(\mathcal{X}) \|^2$ (this is possible since Proposition \ref{prop:weak_estimate_sequence_general} is for an arbitrary $\phi_0$). For $k \geq 0$, we define the following terms recursively:
    \begin{itemize}
        \item $\bar \gamma_{k+1}:=(1-\alpha_k) \gamma_k + \alpha_k \mu$
        \item $\mathcal{V}_{k+1} := \textnormal{Exp}_{\mathcal{Y}_k} \left(\frac{(1-\alpha_k) \gamma_k}{\bar \gamma_{k+1}} \textnormal{Log}_{\mathcal{Y}_k}(\mathcal{V}_k) - \frac{\alpha_k}{B_k \bar \gamma_{k+1}} \textnormal{grad}f(\mathcal{Y}_k) \right)$ 
        \item $\phi_{k+1}^* := (1-\alpha_k) \phi_k^* + \alpha_k f(\mathcal{Y}_k) - \frac{\alpha_k^2}{2 B_k^2 \bar \gamma_{k+1}} \| \textnormal{grad}f(\mathcal{Y}_k) \|^2 \\ + \frac{\alpha_k(1-\alpha_k) \gamma_k}{\bar \gamma_{k+1}} \left( \frac{\mu}{2} \textnormal{dist}^2(\mathcal{Y}_k,\mathcal{V}_k) + \frac{1}{B_k} \langle \textnormal{grad}f(\mathcal{Y}_k), \textnormal{Log}_{\mathcal{Y}_k}(\mathcal{V}_k) \rangle \right)$ 
    \end{itemize}
    If $\gamma_{k+1}$ is chosen such that
    \[
    \gamma_{k+1} \| \textnormal{Log}_{\mathcal{Y}_{k+1}}(\mathcal{V}_{\alpha}) - \textnormal{Log}_{\mathcal{Y}_{k+1}}(\mathcal{V}_{k+1}) \|^2 \leq \bar \gamma_{k+1} \| \textnormal{Log}_{\mathcal{Y}_k}(\mathcal{V}_{\alpha}) - \textnormal{Log}_{\mathcal{Y}_k}( \mathcal{V}_{k+1}) \|^2,
    \]
    then the pair of sequences $(\tau_k,\phi_k)$ defined by
    \begin{align*}
        &\phi_k(\mathcal{X}) := \phi_k^* + \frac{\gamma_k}{2} \| \textnormal{Log}_{\mathcal{Y}_k}(\mathcal{X}) - \textnormal{Log}_{\mathcal{Y}_k}(\mathcal{V}_k) \|^2 \\
        &\tau_0=1, \hspace{2mm} \tau_{k+1} :=(1-\alpha_k) \tau_k
    \end{align*}
    is a weak estimate sequence for $f$.
\end{proposition}

\begin{proof}
    We firstly prove that if $\phi_k(\mathcal{X})= \phi_k^*+\frac{\gamma_k}{2} \| \textnormal{Log}_{\mathcal{Y}_k}(\mathcal{X}) - \textnormal{Log}_{\mathcal{Y}_k}(\mathcal{V}_k) \|^2$, then $ \bar \phi_{k+1}(\mathcal{X})=\phi_{k+1}^* + \frac{\bar \gamma_{k+1}}{2} \| \textnormal{Log}_{\mathcal{Y}_k}(\mathcal{X}) - \textnormal{Log}_{\mathcal{Y}_k}(\mathcal{V}_{k+1}) \|^2$, where $\bar \phi_{k+1}$ is defined recursively from $\phi_k$ as in Proposition \ref{prop:weak_estimate_sequence_general}. 
    
    Indeed, we have
\begin{align*}
    \bar \phi_{k+1}(\mathcal{X})= & (1-\alpha_k) \phi_k(\mathcal{X})+\alpha_k \left(f(\mathcal{Y}_k)+\frac{1}{B_k} \langle \textnormal{grad}f(\mathcal{Y}_k), \textnormal{Log}_{\mathcal{Y}_k}(\mathcal{X}) \rangle + \frac{\mu}{2} \textnormal{dist}^2(\mathcal{Y}_k,\mathcal{X}) \right),
\end{align*}
by its definition in Proposition \ref{prop:weak_estimate_sequence_general}.
We can rewrite the right hand side as
\begin{align*}
& (1-\alpha_k) \phi_k(\mathcal{X})+\alpha_k \left(f(\mathcal{Y}_k)+\frac{1}{B_k} \langle \textnormal{grad}f(\mathcal{Y}_k), \textnormal{Log}_{\mathcal{Y}_k}(\mathcal{X}) \rangle + \frac{\mu}{2} \textnormal{dist}^2(\mathcal{Y}_k,\mathcal{X}) \right) = \\ &  
    (1-\alpha_k) \left(\phi_k^*+\frac{\gamma_k}{2} \| \textnormal{Log}_{\mathcal{Y}_k}(\mathcal{X}) - \textnormal{Log}_{\mathcal{Y}_k}(\mathcal{V}_k) \|^2 \right) \\ & +\alpha_k \left(f(\mathcal{Y}_k)+\frac{1}{B_k} \langle \textnormal{grad}f(\mathcal{Y}_k), \textnormal{Log}_{\mathcal{Y}_k}(\mathcal{X}) \rangle + \frac{\mu}{2} \textnormal{dist}^2(\mathcal{Y}_k,\mathcal{X}) \right),
\end{align*}
where we use the induction hypothesis for $\phi_k$.
By rearranging the terms and completing the square, we can write
\begin{align*}
    &(1-\alpha_k) \left(\phi_k^*+\frac{\gamma_k}{2} \| \textnormal{Log}_{\mathcal{Y}_k}(\mathcal{X}) - \textnormal{Log}_{\mathcal{Y}_k}(\mathcal{V}_k) \|^2 \right) \\ & +\alpha_k \left(f(\mathcal{Y}_k)+\frac{1}{B_k} \langle \textnormal{grad}f(\mathcal{Y}_k), \textnormal{Log}_{\mathcal{Y}_k}(\mathcal{X}) \rangle + \frac{\mu}{2} \textnormal{dist}^2(\mathcal{Y}_k,\mathcal{X}) \right) \\ &= \frac{\bar \gamma_{k+1}}{2} \| \textnormal{Log}_{\mathcal{Y}_k}(\mathcal{X}) \|^2 + \langle \frac{\alpha_k}{B_k} \textnormal{grad}f(\mathcal{Y}_k) - (1-\alpha_k) \gamma_k \textnormal{Log}_{\mathcal{Y}_k}(\mathcal{V}_k) , \textnormal{Log}_{\mathcal{Y}_k}(\mathcal{X}) \rangle  \\ & + (1-\alpha_k) \left(\phi_k^* + \frac{\gamma_k}{2} \| \textnormal{Log}_{\mathcal{Y}_k}(\mathcal{V}_k) \|^2 \right) + \alpha_k f(\mathcal{Y}_k) \\  = & 
        \frac{\bar \gamma_{k+1}}{2} \left \| \textnormal{Log}_{\mathcal{Y}_k}(\mathcal{X}) - \left( \frac{(1-\alpha_k) \gamma_k}{\bar \gamma_{k+1}} \textnormal{Log}_{\mathcal{Y}_k}(\mathcal{V}_k) - \frac{\alpha_k}{B_k \bar \gamma_{k+1}} \textnormal{grad}f(\mathcal{Y}_k) \right) \right \|^2 \\ & - \frac{\bar \gamma_{k+1}}{2} \left \| \frac{(1-\alpha_k) \gamma_k}{\bar \gamma_{k+1}} \textnormal{Log}_{\mathcal{Y}_k}(\mathcal{V}_k) - \frac{\alpha_k}{B_k \bar \gamma_{k+1}} \textnormal{grad}f(\mathcal{Y}_k) \right \|^2 \\ & + (1-\alpha_k) \left(\phi_k^* + \frac{\gamma_k}{2} \| \textnormal{Log}_{\mathcal{Y}_k}(\mathcal{V}_k) \|^2 \right) + \alpha_k f(\mathcal{Y}_k).
\end{align*}
Plugging in the definition of $\mathcal{V}_{k+1}$ and splitting the norm in the second summand, we can write the last expression as
\begin{align*}
    & \frac{\bar \gamma_{k+1}}{2} \left \| \textnormal{Log}_{\mathcal{Y}_k}(\mathcal{X}) - \textnormal{Log}_{\mathcal{Y}_k}(\mathcal{V}_{k+1}) \right \|^2 - \frac{\alpha_k^2}{2 B_k^2 \bar \gamma_{k+1}} \| \textnormal{grad}f(\mathcal{Y}_k) \|^2 \\ & + \frac{\alpha_k(1-\alpha_k) \gamma_k}{\bar \gamma_{k+1}} \left( \frac{\mu}{2} \| \textnormal{Log}_{\mathcal{Y}_k}(\mathcal{V}_k) \|^2 + \frac{1}{B_k} \langle \textnormal{Log}_{\mathcal{Y}_k}(\mathcal{V}_k) , \textnormal{grad}f(\mathcal{Y}_k) \rangle \right) \\ & + (1-\alpha_k) \phi_k^* + \alpha_k f(\mathcal{Y}_k).
\end{align*}
Finally, we can use the definition of $\phi_{k+1}^*$ from $\phi_k^*$ and rewrite again in the desired form:
\begin{align*}
\phi_{k+1}^* + \frac{\bar \gamma_{k+1}}{2} \left \| \textnormal{Log}_{\mathcal{Y}_k}(\mathcal{X}) - \textnormal{Log}_{\mathcal{Y}_k}(\mathcal{V}_{k+1}) \right \|^2.
\end{align*}
This proves our first claim.

    Using the previous computation and the definition of $\gamma_{k+1}$, we immediately get 
    \begin{equation*}
        \phi_{k+1}(\mathcal{V}_{\alpha}) \leq \bar \phi_{k+1}(\mathcal{V}_{\alpha}).
    \end{equation*}
Thus, we have shown that the sequence $\phi_k$ defined as 
\begin{equation*}
    \phi_k(\mathcal{X}) := \phi_k^* + \frac{\gamma_k}{2} \| \textnormal{Log}_{\mathcal{Y}_k}(\mathcal{X}) - \textnormal{Log}_{\mathcal{Y}_k}(\mathcal{V}_k) \|^2
\end{equation*}
satisfies all the assumptions of Proposition \ref{prop:weak_estimate_sequence_general}. 
We therefore conclude that $(\tau_k,\phi_k)$ is a weak estimate sequence.
\end{proof}

\vspace{4mm}

Now we have a concrete definition of $\mathcal{V}_{k+1}$ from $\mathcal{Y}_k$ and $\mathcal{V}_k$. It remains to define $\mathcal{Y}_k$ through $\mathcal{X}_k$ and $\mathcal{V}_k$ and $\mathcal{X}_{k+1}$ through $\mathcal{Y}_k$. We do this with criterion to guarantee $f(\mathcal{X}_k) \leq \phi_k^*$. In order to guarantee that, let us assume that $f(\mathcal{X}_k) \leq \phi_k^*$ and see what happens with $\phi_{k+1}^*$ (towards having an induction step). Using the definition of $\phi_{k+1}^*$ in Proposition \ref{prop:weak_estimate_sequence_special} and $f(\mathcal{X}_k) \leq \phi_k^*$, we have:
\begin{align*}
\label{ineq:general}
    \phi_{k+1}^* & \geq (1-\alpha_k) f(\mathcal{X}_k) + \alpha_k f(\mathcal{Y}_k)-\frac{\alpha_k^2}{2 B_k^2 \bar \gamma_{k+1}} \| \textnormal{grad}f(\mathcal{Y}_k) \|^2 \\ & + \frac{\alpha_k(1-\alpha_k) \gamma_k}{B_k \bar \gamma_{k+1}} \langle \textnormal{grad}f(\mathcal{Y}_k), \textnormal{Log}_{\mathcal{Y}_k}(\mathcal{V}_k) \rangle.
\end{align*}
The usual way to proceed from here is to assume that $f$ is geodesically convex and linearize it from below (see \cite{zhang2018towards}, Lemma 5). Since the Rayleigh quotient on Grassmann is only-locally convex, we employ a different strategy using a geodesic search as in \cite{alimisis2021momentum}, inspired by \cite{nesterov2020primal}. Namely, if we find a way to choose $\mathcal{Y}_k$ from $\mathcal{X}_k$ and $\mathcal{V}_k$ such that
\begin{equation}
\label{ineq:line_search}
    f(\mathcal{X}_k)+ \frac{\alpha_k\gamma_k}{B_k \bar \gamma_{k+1}} \langle \textnormal{grad}f(\mathcal{Y}_k), \textnormal{Log}_{\mathcal{Y}_k}(\mathcal{V}_k) \rangle \geq f(\mathcal{Y}_k)
\end{equation}
then the previous inequality can be reduced to
\begin{align}
\label{ineq:grad_step}
    \phi_{k+1}^* \geq f(\mathcal{Y}_k) - \frac{\alpha_k^2}{2 B_k^2 \bar \gamma_{k+1}} \| \textnormal{grad}f(\mathcal{Y}_k) \|^2.
\end{align}

Inequality (\ref{ineq:line_search}) is satisfied if we choose $\mathcal{Y}_k$ through an exact search in the geodesic connecting $\mathcal{V}_k$ and $\mathcal{X}_k$:
\begin{lemma}
\label{le:geod_search}
    Let
    \begin{equation*}
        \mathcal{Y}_k:=\textnormal{Exp}_{\mathcal{V}_k}(\beta_k \textnormal{Log}_{\mathcal{V}_k}(\mathcal{X}_k))
    \end{equation*}
    where
    \begin{equation*}
        \beta_k := \textnormal{argmin}_{\beta \in [0,1]} (\textnormal{Exp}_{\mathcal{V}_k}(\beta \textnormal{Log}_{\mathcal{V}_k}(\mathcal{X}_k))).
    \end{equation*}
    Then we have
    \begin{equation*}
        f(\mathcal{Y}_k) \leq f(\mathcal{X}_k) \hspace{3mm} \text{and} \hspace{3mm} \langle\textnormal{grad}f(\mathcal{Y}_k), \textnormal{Log}_{\mathcal{Y}_k}(\mathcal{V}_k) \rangle \geq 0.
    \end{equation*}
\end{lemma}
\begin{proof}
    See Appendix B in \cite{alimisis2021momentum}.
\end{proof}

\vspace{2mm}

Thus, if we choose $\mathcal{Y}_k$ in the manner of Lemma \ref{le:geod_search}, the initial inequality for $\phi_{k+1}^*$ implies inequality \ref{ineq:grad_step}. Inequality \ref{ineq:grad_step} is similar to the function value reduction that is obtained via a gradient step:

If we choose 
\begin{equation}
\label{eq:grad_step}
    \mathcal{X}_{k+1}= \Exp_{\mathcal{Y}_k} \left(-\frac{1}{\gamma} \textnormal{grad}f(\mathcal{Y}_k) \right),
\end{equation}
then \ref{eq:func_val_decrease} implies that if we also choose $\alpha_k$ such that
\begin{equation*}
    \frac{\alpha_k^2}{2 B_k^2 \bar \gamma_{k+1}} = \frac{1}{2 \gamma},
\end{equation*}
then 
\begin{equation}
\label{eq:func_val_desc}
   f(\mathcal{X}_{k+1}) \leq f(\mathcal{Y}_k) - \frac{\alpha_k^2}{2 B_k^2 \bar \gamma_{k+1}} \| \textnormal{grad}f(\mathcal{Y}_k) \|^2 
\end{equation}
and consequently
\begin{equation*}
   f(\mathcal{X}_{k+1}) \leq \phi_{k+1}^*.
\end{equation*}

We have also the freedom to make the gradient step from $\mathcal{Y}_k$ following the QR-retraction and choosing the step-size via an exact line-search:

\begin{equation}
\label{eq:grad_step_retr}
    \mathcal{X}_{k+1}=\textnormal{Retr}_{\mathcal{Y}_k} \left(-\eta_{\text{optimal}} \cdot \textnormal{grad}f(\mathcal{Y}_k) \right).
\end{equation}

Similarly, Proposition \ref{prop:func_val_decrease_retr} implies that if we choose $\alpha_k$ such that
\begin{equation*}
    \frac{\alpha_k^2}{2 B_k^2 \bar \gamma_{k+1}} = \frac{2}{5 \gamma} =: \frac{1}{2 \tilde \gamma},
\end{equation*}
then again 
\begin{equation*}
   f(\mathcal{X}_{k+1}) \leq \phi_{k+1}^*
\end{equation*}
is guaranteed. Here $\tilde \gamma$ is defined as $\frac{5}{4} \gamma$, where $\gamma$ is the smoothness constant.

Choosing $\phi_0^*=f(\mathcal{X}_0)$, we can now prove by induction that the following algorithm produces iterates $\mathcal{X}_k$, such that $f(\mathcal{X}_k) \leq \phi_k^*$, where $\phi_k^*$ is defined recursively as in Proposition \ref{prop:weak_estimate_sequence_special}. This analysis yields us naturally to an algorithm, which can be proved to have accelerated convergence guarantees. We choose to write the algorithm using equation \ref{eq:grad_step} to perform the gradient step for obtaining $\mathcal{X}_{k+1}$ from $\mathcal{Y}_k$, but we could also use equation \ref{eq:grad_step_retr}. The only thing that changes is the constant $\gamma$ to $\Tilde{\gamma}$.

Significant effort needs to be spent in proving that all operations in Algorithm \ref{alg:first_version} are well-defined. For that to happen, we need to insure that there is a unique geodesic connecting $\mathcal{V}_k$ and $\mathcal{X}_k$, that $\| \textnormal{grad}f(\mathcal{Y}_k) \|_2 < \frac{\pi}{2}$ and that $\left\|\frac{(1-\alpha_k) \gamma_k}{\bar \gamma_{k+1}} \textnormal{Log}_{\mathcal{Y}_k}(\mathcal{V}_k) - \frac{2 \alpha_k}{\bar \gamma_{k+1}} \textnormal{grad}f(\mathcal{Y}_k) \right\|_2 < \frac{\pi}{2}$ (i.e. these tangent vectors are inside the injecitivity domain, recall equation (\ref{eq:inj_exp})). To guarantee these bounds, a careful selection of hyperparameters is crucial. Algorithm \ref{alg:first_version} is written in a form, in which it is not clear whether some steps are doable, for instance it is not clear whether one can find a $\gamma_{k+1}$ from $\bar \gamma_{k+1}$ such that the requirements of step 9 are satisfied. For the moment we assume that all these can be done and we show in the next section that a careful selection of $\gamma_0$ indeed yields all the requirements of Algorithm \ref{alg:first_version}.

\begin{algorithm}[H]

\small
\caption{Accelerated Gradient Descent for the Block Rayleigh Quotient}
\label{alg:first_version}
\begin{algorithmic}[1]
      \STATE \text{Initialize at} $\mathcal{X}_0=\mathcal{V}_0 \in \Gr(n,p)$ \text{and choose} $\gamma_0$, such that $\frac{\mu}{2} \leq \gamma_0 \leq \gamma$
      \vspace{2mm}
      \FOR {$k \geq 0$}
      \vspace{2mm}
      \STATE $\beta_k =  \underset{\beta \in [0,1]}{\mathrm{argmin}}
      \left \lbrace f(\textnormal{Exp}_{\mathcal{V}_k} (\beta \textnormal{Log}_{\mathcal{V}_k}(\mathcal{X}_k))) \right \rbrace$
      \vspace{2mm}
      \STATE $\mathcal{Y}_k=\textnormal{Exp}_{\mathcal{V}_k}(\beta_k \textnormal{Log}_{\mathcal{V}_k}(\mathcal{X}_k))$
      \vspace{2mm}
      \STATE $4 \alpha_k^2=\frac{(1-\alpha_k) \gamma_k+\alpha_k \mu}{\gamma}$
      \vspace{2mm}
      \STATE $\mathcal{X}_{k+1}= \Exp_{\mathcal{Y}_k} \left(-\frac{1}{\gamma} \textnormal{grad}f(\mathcal{Y}_k) \right)$
      \vspace{2mm}
      \STATE $\bar \gamma_{k+1} = (1-\alpha_k) \gamma_k+ \alpha_k \mu$
      \vspace{2mm}
    
      \STATE $\mathcal{V}_{k+1}= \textnormal{Exp}_{\mathcal{Y}_k} \left(\frac{(1-\alpha_k) \gamma_k}{\bar \gamma_{k+1}} \textnormal{Log}_{\mathcal{Y}_k}(\mathcal{V}_k) - \frac{2 \alpha_k}{\bar \gamma_{k+1}} \textnormal{grad}f(\mathcal{Y}_k) \right)$ 
      \vspace{2mm}
      \STATE $\gamma_{k+1} \| \textnormal{Log}_{\mathcal{Y}_{k+1}}(\mathcal{V}_{\alpha}) - \textnormal{Log}_{\mathcal{Y}_{k+1}}(\mathcal{V}_{k+1}) \|^2 \leq \bar \gamma_{k+1} \| \textnormal{Log}_{\mathcal{Y}_k}(\mathcal{V}_{\alpha}) - \textnormal{Log}_{\mathcal{Y}_k}(\mathcal{V}_{k+1}) \|^2$, such that $\bar \gamma_{k+1} \geq \gamma_{k+1} \geq \mu/2$.
      \vspace{2mm}
      \ENDFOR
\end{algorithmic}      
 \end{algorithm}

\paragraph{Remark} Notice that Algorithm \ref{alg:first_version} is invariant under shifts of the matrix $A$ with multiples of the identity matrix. Indeed, the only steps that are affected by such a shift are steps 3,6 and 8, which feature the function $f$ or its gradient. The shift $A+\alpha I$ yields a function value which is just shifted by a constant, thus step 3 remains unchanged. Also, the Riemannian gradient remains exactly the same as the orthogonal projection neutralizes the extra term obtained by the shift. Thus, steps 6 and 8 also remain unchanged. Notice that the parameters $\gamma$ and $\mu$ remain unchanged as well.

 \begin{theorem}
 \label{thm_local}
     If $\mathcal{X}_0$ satisfies
     \begin{equation*}
         \textnormal{dist}(\mathcal{X}_0,\mathcal{V}_{\alpha}) \leq \frac{1}{8} \sqrt{c_Q} \left(\frac{\delta}{\gamma} \right)^{3/4},
     \end{equation*} 
     and step 9 in Algorithm~\ref{alg:first_version} can be satisfied, i.e. there is $\gamma_{k+1}$ satisfying the required bounds,
     then the following holds:
     \begin{itemize}
         \item[(i)] $f(\mathcal{X}_k), f(\mathcal{Y}_k) \leq f(\mathcal{X}_0)$, for all $k \geq 0$
         \item[(ii)] a$(\mathcal{Y}_k) \geq \frac{1}{2}$
         \item[(iii)] The operations in steps 3, 4, 6,and 8 in Algorithm~\ref{alg:first_version}  are well-defined in the sense that the related tangent vectors are inside the injectivity domain of $\Gr(n,p)$
         \item[(iv)] $f(\mathcal{X}_k) \leq \phi_k^*$, for all $k \geq 0$, where $\phi_k^*$ is defined as
     \begin{align*}
         &\phi_0^*=f(\mathcal{X}_0) \\
         &\phi_{k+1}^* = (1-\alpha_k) \phi_k^* + \alpha_k f(\mathcal{Y}_k) - \frac{\alpha_k^2}{8 \bar \gamma_{k+1}} \| \textnormal{grad}f(\mathcal{Y}_k) \|^2 \\ & + \frac{\alpha_k(1-\alpha_k) \gamma_k}{\bar \gamma_{k+1}} \left( \frac{\mu}{2} \textnormal{dist}^2(\mathcal{Y}_k,\mathcal{V}_k) + 2 \langle \textnormal{grad}f(\mathcal{Y}_k), \textnormal{Log}_{\mathcal{Y}_k}(\mathcal{V}_k) \rangle \right).
     \end{align*}
     \end{itemize}

 \end{theorem}

\begin{proof}
We proceed to the proof of all  points together by induction. 

For $k=0$, the first holds trivially since $\mathcal{X}_0=\mathcal{Y}_0$.

The second  point holds, because
\begin{equation*}
    a(\mathcal{Y}_0) \geq \cos(\theta_{max}(\mathcal{Y}_0,\mathcal{V}_{\alpha})) \geq \cos(\textnormal{dist}(\mathcal{Y}_0,\mathcal{V}_{\alpha})) \geq \cos\left(1 \right) > \frac{1}{2}.
\end{equation*}
Here $\theta_{max}$ is used to denote the biggest principal angles between subspaces. This inequality implies that $B_0$ can be chosen to be $\frac{1}{2}$.

The third holds since $\mathcal{X}_0 = \mathcal{V}_0$ (steps 3-4 are well-defined) and by $\gamma$-smoothness of $f$, we have
\begin{equation*}
    \| \textnormal{grad}f(\mathcal{Y}_0) \| \leq \gamma \textnormal{dist}(\mathcal{Y}_0, \mathcal{V}_{\alpha}) = \gamma \textnormal{dist}(\mathcal{X}_0, \mathcal{V}_{\alpha}) \leq \frac{\gamma}{8} \left(\frac{\delta}{\gamma} \right)^{3/4} \leq \frac{\gamma}{8}.
\end{equation*}
This implies that the biggest singular value of $-\frac{1}{\gamma} \textnormal{grad}f(\mathcal{Y}_0)$ is less than $\frac{\pi}{2}$, thus $- \frac{1}{\gamma} \textnormal{grad}f(\mathcal{Y}_0)$ is inside the injectivity domain and step 6 is well-defined for $k=0$. For step 8, we have that $\Log_{\mathcal{Y}_0}(\mathcal{V}_0)=0$, thus we need to bound
$
    \left\| -\frac{2 \alpha_0}{\bar \gamma_1} \textnormal{grad}f(\mathcal{Y}_0) \right \|_2
$. For that, we use inequality (\ref{eq:func_val_desc}), which can be rewritten (with $B_0=\frac{1}{2}$) as
\begin{equation*}
   \frac{2\alpha_{0}^2}{\bar \gamma_1} \| \textnormal{grad}f(\mathcal{Y}_{0}) \|^2 \leq f(\mathcal{Y}_{0})-f(\mathcal{X}_{1}) \leq f(\mathcal{X}_0)-f^*.
\end{equation*}
By multiplying both sides of the previous inequality with $\frac{2}{\bar \gamma_1}$ and $\gamma$-smoothness of $f$, we get
\begin{equation*}
    \frac{4\alpha_{0}^2}{\bar \gamma_1^2} \| \textnormal{grad}f(\mathcal{Y}_{0}) \|^2 \leq \frac{\gamma}{\bar \gamma_1} \textnormal{dist}^2(\mathcal{X}_0,\mathcal{V}_{\alpha}) \leq 2 \frac{\gamma}{\mu} \frac{1}{64} c_Q \left(\frac{\delta}{\gamma} \right)^\frac{3}{2} = \frac{1}{64} \frac{\gamma}{\delta} \left(\frac{\delta}{\gamma} \right)^\frac{3}{2} \leq 1.
\end{equation*}
The second inequality follows from $\bar \gamma_1 \geq \gamma_0 \geq \frac{\mu}{2}$.

This bound implies that $-\frac{2 \alpha_0}{\bar \gamma_1} \textnormal{grad}f(\mathcal{Y}_0)$ is inside the injectivity domain and step 8 is well-defined for $k=0$.

The fourth  point holds trivially since $\phi_0^*$ is defined as $f(\mathcal{X}_0)$. 

Now, we assume that all the points hold for all iterations up to iteration $k$ and wish to prove that they still hold for $k+1$.

By the construction of the algorithm we have
\begin{equation*}
    f(\mathcal{Y}_{k+1}) \leq f(\mathcal{Y}_k).
\end{equation*}
This is because $f(\mathcal{Y}_{k+1}) \leq f(\mathcal{X}_{k+1})$ (due to the geodesic search, step 3-4) and $f(\mathcal{X}_{k+1}) \leq f(\mathcal{Y}_k)$ (due to the gradient step, step 6). Thus, we can conclude that $f(\mathcal{Y}_{k+1}) \leq f(\mathcal{Y}_k) \leq f(\mathcal{Y}_0)=f(\mathcal{X}_0)$ by the induction hypothesis. The same inequalities imply that $f(\mathcal{X}_{k+1}) \leq f(\mathcal{X}_k)$ and by the induction hypothesis we have $f(\mathcal{X}_{k+1}) \leq f(\mathcal{X}_0)$. Thus, the first point is correct at iteration $k+1$.

We can use the result of the first point to bound the distance of the iterates $\mathcal{Y}_{k+1},\mathcal{X}_{k+1}$ from the optimum $\mathcal{V}_{\alpha}$, using the quadratic growth condition (Proposition \ref{prop:quadratic_growth}):
\begin{align*}
     \textnormal{dist}^2(\mathcal{Y}_{k+1},\mathcal{V}_{\alpha}) &\leq \frac{1}{c_Q \delta} (f(\mathcal{Y}_{k+1})-f^*) 
     \leq \frac{1}{c_Q \delta} (f(\mathcal{X}_0)-f^*)
     \\ & \leq \frac{\gamma}{2 c_Q \delta} \textnormal{dist}^2(\mathcal{X}_0, \mathcal{V}_{\alpha}) \leq \frac{1}{128} \left(\frac{\delta}{\gamma} \right)^{1/2}.
\end{align*}
The same bound holds also for $\textnormal{dist}(\mathcal{X}_{k+1},\mathcal{V}_{\alpha})$. Thus, we have
\begin{equation*}
    \textnormal{dist}(\mathcal{X}_{k+1},\mathcal{V}_{\alpha}), \textnormal{dist}(\mathcal{Y}_{k+1},\mathcal{V}_{\alpha}) \leq \frac{1}{8\sqrt{2}} \left(\frac{\delta}{\gamma} \right)^{1/4}.
\end{equation*}

This lower bound implies that the quantity $a(\mathcal{Y}_{k+1})$ can be bounded as
\begin{equation*}
    a(\mathcal{Y}_{k+1}) \geq \cos(\theta_{max}(\mathcal{Y}_{k+1},\mathcal{V}_{\alpha})) \geq \cos(\textnormal{dist}(\mathcal{Y}_{k+1},\mathcal{V}_{\alpha})) \geq \cos\left(1 \right) > \frac{1}{2},
\end{equation*}
which implies that the second point holds at the $k+1$ iteration.

The result of the second point together with the induction hypothesis provide that $a(\mathcal{Y}_i)\geq \frac{1}{2}$ for all $i=0,...,k+1$. This means that $B_i$ can be taken equal to $\frac{1}{2}$ in all the analysis of Sections 4 and 5 and with a choice of $\alpha_k$ as in step 5 of the algorithm and $\phi_{k+1}^*$ as defined in the statement of the fourth  point, we have automatically that $f(\mathcal{X}_{k+1}) \leq \phi_{k+1}^*$. Thus the fourth  point is correct at iteration $k+1$. 

For showing that the steps 3-4 and 6 in the algorithm are well-defined in iteration $k+1$ (third  point), we need also a bound for $\textnormal{dist}(\mathcal{V}_{k+1},\mathcal{V}_{\alpha})$, which turns out to be a quite complicted. For that, we start by using the geometric bound stated in Lemma \ref{prop:tangent_space}:

\begin{equation*}
    \textnormal{dist}(\mathcal{V}_{k+1},\mathcal{V}_{\alpha}) \leq \| \Log_{\mathcal{Y}_{k+1}}(\mathcal{V}_{\alpha})-\Log_{\mathcal{Y}_{k+1}}(\mathcal{V}_{k+1}) \|.
\end{equation*}

The quantity on the right hand side is directly related to the sequence $\phi_k^*$:
\begin{align*}
    &\| \Log_{\mathcal{Y}_{k+1}}(\mathcal{V}_{\alpha})-\Log_{\mathcal{Y}_{k+1}}(\mathcal{V}_{k+1}) \|^2 = \frac{2}{\gamma_{k+1}} (\phi_{k+1}(\mathcal{V_{\alpha}})-\phi_{k+1}^*) \leq \frac{2}{\gamma_{k+1}} (\phi_0(\mathcal{V}_{\alpha})-f^*) = \\ & \frac{2}{\gamma_{k+1}} \left(\phi_0^*+\frac{\gamma_0}{2} \textnormal{dist}^2(\mathcal{X}_0,\mathcal{V}_{\alpha}) - f^* \right) = \frac{2}{\gamma_{k+1}} \left(f(\mathcal{X}_0)-f^*+\frac{\gamma_0}{2} \textnormal{dist}^2(\mathcal{X}_0,\mathcal{V}_{\alpha}) \right) \leq  \\ &  
    \frac{2}{\gamma_{k+1}} \frac{\gamma+\gamma_0}{2} \textnormal{dist}^2(\mathcal{X}_0,\mathcal{V}_{\alpha}) = \frac{\gamma+\gamma_0}{\gamma_{k+1}}  \textnormal{dist}^2(\mathcal{X}_0,\mathcal{V}_{\alpha}) \leq 4 \frac{\gamma}{\mu} \textnormal{dist}^2(\mathcal{X}_0,\mathcal{V}_{\alpha}) \leq \\ &  = 4 \frac{\gamma}{2 c_Q \delta} \frac{1}{64} c_Q \left(\frac{\delta}{\gamma} \right)^{3/2} = \frac{1}{32} \left( \frac{\delta}{\gamma} \right)^{1/2}.
\end{align*}
The first equality is implied by the definition of $\phi_{k+1}$, the first inequality by Proposition \ref{prop:bound*} combined with the inequality $\phi_{k+1}^* \geq f(\mathcal{X}_{k+1}) \geq f^*$ which holds since we have already explained that the fourth point holds for $k+1$, the second equality by the definition of $\phi_0$, the third equality by the definition of $\phi_0^*$, the second inequality by $\gamma$-smoothness, the third inequality by the upper bound on $\gamma_0$ and the lower bound on $\gamma_{k+1}$, and the rest are simple substitutions involving the bound in the initial distance $\textnormal{dist}(\mathcal{X}_0,\mathcal{V}_{\alpha})$.

Thus 
\begin{equation*}
    \textnormal{dist}(\mathcal{V}_{k+1},\mathcal{V}_{\alpha}) \leq \frac{1}{4 \sqrt{2}} \left( \frac{\delta}{\gamma} \right)^{1/4}.
\end{equation*}

Combining the bounds on $\textnormal{dist}(\mathcal{X}_{k+1},\mathcal{V}_{\alpha})$ and $\textnormal{dist}(\mathcal{V}_{k+1},\mathcal{V}_{\alpha})$ with the triangle inequality, we get
\begin{equation*}
    \textnormal{dist}(\mathcal{X}_{k+1},\mathcal{V}_{k+1}) \leq \textnormal{dist}(\mathcal{X}_{k+1},\mathcal{V}_{\alpha})+\textnormal{dist}(\mathcal{V}_{k+1},\mathcal{V}_{\alpha}) \leq \frac{1}{2}.
\end{equation*}
This implies that there is a unique geodesic connecting $\mathcal{X}_{k+1}$ and $\mathcal{V}_{k+1}$, thus steps 3-4 are well-defined in iteration $k+1$.

In addition, $\gamma$-smoothness of $f$ implies that
\begin{equation*}
    \| \textnormal{grad}f(\mathcal{Y}_{k+1}) \| \leq \gamma \textnormal{dist}(\mathcal{Y}_{k+1},\mathcal{V}_{\alpha}) \leq \frac{\gamma}{8}
\end{equation*}
which provides the bound
\begin{equation*}
    \left \|-\frac{1}{\gamma} \textnormal{grad}f(\mathcal{Y}_{k+1}) \right \|_2 \leq \frac{1}{8}.
\end{equation*}
The last bound implies that $-\frac{1}{\gamma} \textnormal{grad}f(\mathcal{Y}_{k+1})$ is inside the injectivity domain, thus step 6 is well-defined in iteration $k+1$. 

We lastly deal with step 8. We have that $\frac{(1-\alpha_{k+1}) \gamma_{k+1}}{\bar \gamma_{k+2}} = \frac{(1-\alpha_{k+1}) \gamma_{k+1}}{(1-\alpha_{k+1}) \gamma_{k+1}+\alpha_{k+1} \mu} \leq 1$ and $\| \Log_{\mathcal{Y}_{k+1}}(\mathcal{V}_{k+1}) \| \leq \| \Log_{\mathcal{X}_{k+1}}(\mathcal{V}_{k+1}) \| \leq \frac{1}{2}$. For the second summand $\frac{2 \alpha_{k+1}}{\bar \gamma_{k+2}} \textnormal{grad}f(\mathcal{Y}_{k+1})$, we use inequality (\ref{eq:func_val_desc}), which can be rewritten (with $B_{k+1}=\frac{1}{2}$) as
\begin{equation*}
   \frac{2\alpha_{k+1}^2}{\bar \gamma_{k+2}} \| \textnormal{grad}f(\mathcal{Y}_{k+1}) \|^2 \leq f(\mathcal{Y}_{k+1})-f(\mathcal{X}_{k+2}) \leq f(\mathcal{X}_0)-f^*.
\end{equation*}

Multiplying both sides with $\frac{2}{\bar \gamma_{k+2}}$ and using $\gamma$-smoothness of $f$, we get
\begin{equation*}
    \frac{4\alpha_{k+1}^2}{\bar \gamma_{k+2}^2} \| \textnormal{grad}f(\mathcal{Y}_{k+1}) \|^2 \leq \frac{\gamma}{\bar \gamma_{k+2}} \textnormal{dist}^2(\mathcal{X}_0,\mathcal{V}_{\alpha}) 
\end{equation*}

By definition, we have $\frac{\mu}{2} \leq \gamma_{k+1} \leq \bar \gamma_{k+2}$. Plugging in the assumed bound on the initial distance, we get
\begin{equation*}
    \frac{4\alpha_{k+1}^2}{\bar \gamma_{k+2}^2} \| \textnormal{grad}f(\mathcal{Y}_{k+1}) \|^2 \leq 2 \frac{\gamma}{\mu} \frac{1}{64} c_Q \left(\frac{\delta}{\gamma} \right)^\frac{3}{2} = \frac{1}{64} \frac{\gamma}{\delta} \left(\frac{\delta}{\gamma} \right)^\frac{3}{2} = \frac{1}{64} \left(\frac{\delta}{\gamma} \right)^\frac{1}{2} \leq \frac{1}{4},
\end{equation*}
where we used that $\mu=2 c_Q \delta$.

By triangle inequality, we get
\begin{equation*}
    \left\|\frac{(1-\alpha_{k+1}) \gamma_{k+1}}{\bar \gamma_{k+2}} \textnormal{Log}_{\mathcal{Y}_{k+1}}(\mathcal{V}_{k+1}) - \frac{2 \alpha_{k+1}}{\bar \gamma_{k+2}} \textnormal{grad}f(\mathcal{Y}_{k+1}) \right\| \leq \frac{1}{2} + \frac{1}{2} = 1,
\end{equation*}
thus
\begin{equation*}
    \frac{(1-\alpha_{k+1}) \gamma_{k+1}}{\bar \gamma_{k+2}} \textnormal{Log}_{\mathcal{Y}_{k+1}}(\mathcal{V}_{k+1}) - \frac{2 \alpha_{k+1}}{\bar \gamma_{k+2}} \textnormal{grad}f(\mathcal{Y}_{k+1})
\end{equation*}
is inside the injectivity radius of $\Gr(n,p)$ and step 8 is well-defined at iteration $k+1$.

With that in order, the simultaneous induction of all four  points is complete.
\end{proof}

\section{Effect of curvature/choice of parameters}

In Algorithm \ref{alg:first_version}, it is not clear whether we can choose $\gamma_{k+1}$ from $\bar \gamma_{k+1}$ (step 9) in a tractable way, such that $\gamma_{k+1} \geq \frac{\mu}{2}$. 

We start by showing that there is a way to choose $\gamma_{k+1}$ from $\bar \gamma_{k+1}$, such that 
\begin{equation*}
    \gamma_{k+1} \| \textnormal{Log}_{\mathcal{Y}_{k+1}}(\mathcal{V}_{\alpha}) - \textnormal{Log}_{\mathcal{Y}_{k+1}}(\mathcal{V}_{k+1}) \|^2 \leq \bar \gamma_{k+1} \| \textnormal{Log}_{\mathcal{Y}_k}(\mathcal{V}_{\alpha}) - \textnormal{Log}_{\mathcal{Y}_k}(\mathcal{V}_{k+1}) \|^2.
\end{equation*}
Note that $\mathcal{Y}_{k+1}$ is not yet computed at step 9, but it is to be computed exactly in the next iteration of the algorithm. However, this is not a problem, since the geometric result (Lemma \ref{le:geom_bound}) holds for any four points on a manifold of bounded sectional curvatures.

\begin{proposition}
\label{prop:param_curv}
    Choose 
    \begin{equation*}
        \gamma_0 \geq \frac{\sqrt{\beta^2+(1+\beta)\frac{\mu}{\gamma}}-\beta}{\sqrt{\beta^2+(1+\beta)\frac{\mu}{\gamma}}+\beta} \cdot \mu
    \end{equation*}
    and $\gamma_0 \leq \gamma$, 
    where
    \begin{equation*}
        \beta = \frac{1}{5} \sqrt{\frac{\mu}{\gamma}}.
    \end{equation*}
If $\mathcal{X}_0$ satisfies
     \begin{equation*}
         \textnormal{dist}(\mathcal{X}_0,\mathcal{V}_{\alpha}) \leq \frac{1}{8} \sqrt{c_Q} \left(\frac{\delta}{\gamma} \right)^{3/4},
     \end{equation*} 
then one can choose $\gamma_{k+1}$ from $\bar \gamma_{k+1}$ satisfying all the requirements in step 9 of Algorithm \ref{alg:first_version}, as
\begin{equation*}
    \gamma_{k+1} = \frac{1}{1+\beta} \bar \gamma_{k+1}.
\end{equation*}
\end{proposition}

\begin{proof}
We proceed by induction. For $k=0$, $\gamma_0$ satisfies trivially the main inequality of step 9 (because there is no $\bar \gamma_0$). Also $\gamma_0 \geq \frac{\mu}{2}$, since (easy to see)
\begin{equation*}
    \frac{\sqrt{\beta^2+(1+\beta)\frac{\mu}{\gamma}}-\beta}{\sqrt{\beta^2+(1+\beta)\frac{\mu}{\gamma}}+\beta} \geq \frac{1}{2},
\end{equation*}
if $\beta=\frac{1}{5} \sqrt{\frac{\mu}{\gamma}}$.

Now we assume that we can choose $\gamma_{k+1}$ as in step 9 of Algorithm \ref{alg:first_version} in the first $k$ iterations. Then the result of Theorem \ref{thm_local} holds and its proof guarantees that 
\begin{equation*}
    \textnormal{dist}(\mathcal{Y}_k,\mathcal{V}_{\alpha}),(\mathcal{Y}_{k+1},\mathcal{V}_{\alpha}) \leq \frac{1}{8\sqrt{2}} \left(\frac{\delta}{\gamma} \right)^{1/4}.
\end{equation*}

This implies the weaker inequality
\begin{equation*}
 \textnormal{dist}(\mathcal{Y}_k,\mathcal{V}_{\alpha}), \textnormal{dist}(\mathcal{Y}_{k+1},\mathcal{V}_{\alpha}) \leq \frac{1}{4 \sqrt{K}}, 
\end{equation*}
where $K$ is an upper bound of the sectional curvatures of the Grassmann manifold and it is taken equal to $2$. Thus, the points $\mathcal{Y}_k,\mathcal{Y}_{k+1},\mathcal{V}_{k+1}$ and $\mathcal{V}_{\alpha}$ satisfy the assumptions of Lemma \ref{le:geom_bound}. This gives

\begin{equation*}
    \| \textnormal{Log}_{\mathcal{Y}_{k+1}}(\mathcal{V}_{\alpha}) - \textnormal{Log}_{\mathcal{Y}_{k+1}}(\mathcal{V}_{k+1}) \|^2 \leq \left(1+10 \frac{1}{128} \left(\frac{\delta}{\gamma} \right)^{1/2} \right) \| \textnormal{Log}_{\mathcal{Y}_k}(\mathcal{V}_{\alpha}) -\textnormal{Log}_{\mathcal{Y}_k}(\mathcal{V}_{k+1}) \|^2.    
\end{equation*}

Thus, we can choose $\gamma_{k+1}$ from $\bar \gamma_{k+1}$ such that
\begin{equation*}
    \frac{\bar \gamma_{k+1}}{\gamma_{k+1}} \geq 1+\frac{10}{128} \left(\frac{\delta}{\gamma} \right)^{1/2}
\end{equation*}
or similarly, we can take $\gamma_{k+1}=\frac{1}{1+\beta} \bar \gamma_{k+1}$ with 
\begin{equation*}
    \beta \leq \frac{10}{128} \left(\frac{\delta}{\gamma} \right)^{1/2}.
\end{equation*}
It easy to see that $\frac{1}{5} \left(\frac{\mu}{\gamma} \right)^{1/2} \leq \frac{10}{128} \left(\frac{\delta}{\gamma} \right)^{1/2}$, thus
\begin{equation*}
    \beta = \frac{1}{5} \left(\frac{\mu}{\gamma} \right)^{1/2}
\end{equation*}
is a valid choice.
Such a selection of $\beta$ is important for the rest. Note that $\beta$ is involved directly in the selection of $\gamma_0$ and affects the sequences $\gamma_k$ and $\bar \gamma_k$.

We now prove with the aforementioned selections of $\beta$ and $\gamma_0$ that $\gamma_{k+1}$ selected in step 9 always satisfies $\gamma_{k+1} \geq \frac{\mu}{2}$.

We first show that if $\gamma_s \geq \frac{\sqrt{\beta^2+(1+\beta)\frac{\mu}{\gamma}}-\beta}{\sqrt{\beta^2+(1+\beta)\frac{\mu}{\gamma}}+\beta} \cdot \mu $, then
$\alpha_s \geq \frac{\sqrt{\beta^2+(1+\beta)\frac{\mu}{\gamma}}-\beta}{2}$. To that end, we use the definition of $\alpha_s$ at step 4 of Algorithm \ref{alg:first_version}:
\begin{equation*}
    4 \alpha_s^2=\frac{(1-\alpha_s) \gamma_s+\alpha_s \mu}{\gamma}.
\end{equation*}
The positive solution of this quadratic equation is
\begin{equation*}
  \alpha_s = \frac{(\mu-\gamma_s) \frac{1}{4 \gamma}+\sqrt{(\mu-\gamma_s)^2 \frac{1}{16 \gamma^2}+\frac{\gamma_s}{\gamma}}}{2} =: g(\gamma_s) .
\end{equation*}
We first note that $\alpha_s$ is always less than $1$. Indeed, this happens if and only if
\begin{equation*}
    (\mu-\gamma_s) \frac{1}{4 \gamma}+\sqrt{(\mu-\gamma_s)^2 \frac{1}{16 \gamma^2}+\frac{\gamma_s}{\gamma}} \leq 2
\end{equation*}
or even stronger if 
\begin{equation*}
    (\mu-\gamma_s) \frac{1}{2 \gamma}+\sqrt{\frac{\gamma_s}{\gamma}} \leq 2
\end{equation*}
which is equivalent with
\begin{equation*}
    d^2 - 2d - \frac{\mu}{\gamma} \geq -4,
\end{equation*}
where $d:= \sqrt{\frac{\gamma_s}{\gamma}}$. Since $\frac{\mu}{\gamma} \leq 1$, it suffices to hold 
\begin{equation*}
    d^2 - 2d +3 \geq 0,
\end{equation*}
which always holds. Now, we bound $\alpha_s$ from below.

The function $g$ is increasing and we have
\begin{equation*}
    \alpha_s \geq g\left(\frac{\sqrt{\beta^2+(1+\beta)\frac{\mu}{\gamma}}-\beta}{\sqrt{\beta^2+(1+\beta)\frac{\mu}{\gamma}}+\beta} \cdot \mu \right)=g\left(\frac{C-\beta}{C+\beta} \cdot \mu \right),
\end{equation*}

where
\begin{equation*}
    C:=\sqrt{\beta^2+(1+\beta)\frac{\mu}{\gamma}}.
\end{equation*}

We have
\begin{equation*}
    \mu-\frac{C-\beta}{C+\beta} \mu= \frac{2 \beta}{C+\beta} \mu = \frac{2 \beta (C- \beta)}{C^2-\beta^2} = \frac{2 \beta (C-\beta)}{(1+\beta) \frac{\mu}{\gamma}} \mu = \frac{2 \gamma \beta (C-\beta)}{(1+\beta)}
\end{equation*}
and
\begin{equation*}
    \frac{C-\beta}{C+\beta} \mu = \frac{(C-\beta)^2}{C^2-\beta^2} \mu = \frac{(C-\beta)^2}{(1+\beta) \frac{\mu}{\gamma}} \mu = \frac{\gamma (C-\beta)^2}{1+\beta}.
\end{equation*}

Then,
\begin{align*}
    g\left(\frac{C-\beta}{C+\beta} \cdot \mu \right)  & = \frac{\frac{2 \gamma \beta (C-\beta)}{(1+\beta)}\frac{1}{4 \gamma}+\sqrt{\frac{4 \gamma^2 \beta^2 (C-\beta)^2}{(1+\beta)^2}\frac{1}{16 \gamma^2}+\frac{\gamma (C-\beta)^2}{1+\beta}\frac{1}{\gamma}}}{2} \\ & = \frac{\frac{ \beta (C-\beta)}{2(1+\beta)}+\sqrt{\frac{ \beta^2 (C-\beta)^2}{4(1+\beta)^2}+\frac{(C-\beta)^2}{1+\beta}}}{2} \\ & = (C-\beta) \frac{\frac{\beta}{2(1+\beta)}+\sqrt{\frac{\beta^2+4\beta+4}{4(1+\beta)^2}}}{2} \\ & = 
    \frac{C-\beta}{2} \left( \frac{\beta}{2(1+\beta)}+\frac{\beta+2}{2(1+\beta)} \right) \\ & =
    \frac{C-\beta}{2}.
\end{align*}
Thus,
\begin{equation*}
    \alpha_s \geq \frac{\sqrt{\beta^2+(1+\beta)\frac{\mu}{\gamma}}-\beta}{2}.
\end{equation*}

Now we prove that 
\begin{equation*}
    \gamma_k \geq \frac{C-\beta}{C+\beta} \cdot \mu 
\end{equation*}
for any $k \geq 0$ by induction.

The claim is correct for $k=0$, by the choice of $\gamma_0$. Let us assume that $\gamma_k \geq \frac{C-\beta}{C+\beta} \cdot \mu$. This also implies that $\alpha_k \geq \frac{C-\beta}{2}$ by the previous argument. Then $\gamma_{k+1}$ satisfies
\begin{equation*}
        (1+\beta) \gamma_{k+1}=(1-\alpha_k) \gamma_k + \alpha_k \mu.
    \end{equation*}

Since $\alpha_k \leq 1$ and $\gamma_k \geq \frac{C-\beta}{C+\beta} \mu$, we have
\begin{align*}
    &(1-\alpha_k) \gamma_k + \alpha_k \mu \geq (1-\alpha_k) \frac{C-\beta}{C+\beta} \mu + \alpha_k \mu = (1-\alpha_k)\mu+\alpha_k \mu - (1-\alpha_k) \frac{2 \beta}{C+\beta} \mu \\ =& \mu+(\alpha_k-1)\frac{2 \beta}{C+\beta} \mu \geq \mu + \left(\frac{C-\beta}{2}-1 \right) \frac{2 \beta}{C+\beta} \mu = \left(1+\frac{(C-\beta) \beta}{C+\beta} - \frac{2 \beta}{C+\beta} \right) \mu \\ =& (1+\beta) \frac{C-\beta}{C+\beta} \mu.
\end{align*}
Thus $\gamma_{k+1} \geq \frac{C-\beta}{C+\beta} \mu$ and the desired result holds.

For proving that $\gamma_k \geq \frac{\mu}{2}$, we only need to show that $\frac{C-\beta}{C+\beta} \geq \frac{1}{2}$, which is quite easy to see. This inequality can be written equivalently as
\begin{equation*}
    2C-2\beta \geq C+\beta \Leftrightarrow C \geq 3 \beta \Leftrightarrow \beta^2+(1+\beta)\frac{\mu}{\gamma} \geq 9 \beta^2 \Leftrightarrow (1+\beta)\frac{\mu}{\gamma} \geq 8 \beta^2 = \frac{8}{25} \frac{\mu}{\gamma}
\end{equation*}
Since $1+\beta>1$, the last inequality holds and the desired result follows. Thus, the sequence $\gamma_k$, created by $\bar \gamma_k$ as in the statement of the proposition, satisfies the requirements of step 9. 
\end{proof}

Proposition \ref{prop:param_curv} leads us to a more concrete version of Algorithm \ref{alg:first_version} with a specific choice of parameters:

\begin{algorithm}[H]

\small
\caption{Accelerated Gradient Descent for the Block Rayleigh Quotient}
\label{alg:second_version}
\begin{algorithmic}[1]
      \STATE \text{Initialize at} $\mathcal{X}_0=\mathcal{V}_0 \in \Gr(n,p)$ \text{and choose shrinkage parameter} $\beta = \frac{1}{5} \sqrt{\frac{\mu}{\gamma}}$
      \vspace{2mm}
      \STATE \text{Choose} $\gamma_0 \geq \frac{\sqrt{\beta^2+(1+\beta)\frac{\mu}{\gamma}}-\beta}{\sqrt{\beta^2+(1+\beta)\frac{\mu}{\gamma}}+\beta} \cdot \mu$
      \vspace{2mm}
      \FOR {$k \geq 0$}
      \vspace{2mm}
      \STATE $\beta_k =  \underset{\eta \in [0,1]}{\mathrm{argmin}}
      \left \lbrace f(\textnormal{Exp}_{\mathcal{V}_k} (\eta \textnormal{Log}_{\mathcal{V}_k}(\mathcal{X}_k))) \right \rbrace$
      \vspace{2mm}
      \STATE $\mathcal{Y}_k=\textnormal{Exp}_{\mathcal{V}_k}(\beta_k \textnormal{Log}_{\mathcal{V}_k}(\mathcal{X}_k))$
      \vspace{2mm}
      \STATE $4 \alpha_k^2=\frac{(1-\alpha_k) \gamma_k+\alpha_k \mu}{\gamma}$
      \vspace{2mm}
      \STATE $\mathcal{X}_{k+1}= \Exp_{\mathcal{Y}_k} \left(-\frac{1}{\gamma} \textnormal{grad}f(\mathcal{Y}_k) \right)$
      \vspace{2mm}
      \STATE $\bar \gamma_{k+1} = (1-\alpha_k) \gamma_k+ \alpha_k \mu$
      \vspace{2mm}
      \STATE $\gamma_{k+1} = \frac{1}{1+\beta} \bar \gamma_{k+1}$
      \vspace{2mm}
      \STATE $\mathcal{V}_{k+1}= \textnormal{Exp}_{\mathcal{Y}_k} \left(\frac{(1-\alpha_k) \gamma_k}{\bar \gamma_{k+1}} \textnormal{Log}_{\mathcal{Y}_k}(\mathcal{V}_k) - \frac{2 \alpha_k}{\bar \gamma_{k+1}} \textnormal{grad}f(\mathcal{Y}_k) \right)$ 
      \vspace{2mm}
      
      \ENDFOR
\end{algorithmic}      
 \end{algorithm}

From now on, we use Algorithm \ref{alg:second_version} as our standard algorithm. Its convergence is analysed in the next section.

\section{Convergence}

We are finally ready to complete the convergence analysis. We start with the following simple result. 

\begin{proposition}
\label{prop:first_conv}
    The sequence $\mathcal{X}_k$ generated by Algorithm \ref{alg:second_version} satisfies
    \begin{equation*}
        f(\mathcal{X}_k)-f^* \leq \tau_k \left(f(\mathcal{X}_0)-f^*+\frac{\gamma_0}{2} \textnormal{dist}^2(\mathcal{X}_0,\mathcal{V}_{\alpha}) \right).
    \end{equation*}
\end{proposition}

\begin{proof}
    We choose
    \begin{equation*}
        \phi_0(\mathcal{X})=f(\mathcal{X}_0)+\frac{\gamma_0}{2} \textnormal{dist}^2(\mathcal{X}_0,\mathcal{V}_{\alpha})
    \end{equation*}
    as the beginning of the estimate sequence.
    Then $\phi_0^*=f(\mathcal{X}_0)$ and by construction of $\phi_k^*$ in Theorem \ref{thm_local}, we get $f(\mathcal{X}_k) \leq \phi_k^*$. The result now follows by simply applying Proposition \ref{prop:bound*}.
\end{proof}

Proposition \ref{prop:first_conv} provides a worst-case upper bound for the sub-optimality of $f$ and it only remains to estimate $\tau_k$. Such an estimation can be easilty obtained by the proof of Proposition \ref{prop:param_curv}.
\begin{proposition}
\label{prop:main_conv}
The sequence $\tau_k$, which is generated by the $\alpha_k$ quantites of Algorithm \ref{alg:second_version} starting from a point $\mathcal{X}_0$ such that
     \begin{equation*}
         \textnormal{dist}(\mathcal{X}_0,\mathcal{V}_{\alpha}) \leq \frac{1}{8} \sqrt{c_Q} \left(\frac{\delta}{\gamma} \right)^{3/4},
     \end{equation*}  
    is upper bounded as 
    \begin{equation*}
        \tau_k \leq \left(1-\frac{2}{5} \sqrt{\frac{\mu}{\gamma}} \right)^k.
    \end{equation*}
\end{proposition}

\begin{proof}
By Proposition \ref{prop:weak_estimate_sequence_general}, $\tau_k$ is defined recursively as $\tau_0=1$ and $\tau_{k+1}=(1-\alpha_k) \tau_k$. This implies that
\begin{equation*}
    \tau_k = \Pi_{i=0}^{k-1} (1-\alpha_i)
\end{equation*}
and we only need to estimate a lower bound for $\alpha_i$.

The proof of Proposition \ref{prop:param_curv} provides a lower bound for $\alpha_i$ as
\begin{equation*}
    \alpha_i \geq \frac{\sqrt{\beta^2+(1+\beta)\frac{\mu}{\gamma}}-\beta}{2},
\end{equation*}
with $\beta=\frac{1}{5} \left(\frac{\mu}{\gamma} \right)^{1/2}$. This is because we proved that
$\gamma_i \geq \frac{C-\beta}{C+\beta} \cdot \mu $, for all $i$ by induction and also proved that
\begin{equation*}
    \gamma_i \geq \frac{C-\beta}{C+\beta} \cdot \mu \Rightarrow \alpha_i \geq \frac{\sqrt{\beta^2+(1+\beta)\frac{\mu}{\gamma}}-\beta}{2}.
\end{equation*}

Taking into account the exact value for $\beta$, we can rewrite the lower bound for $\alpha_i$ as 
\begin{equation*}
    \alpha_i \geq \frac{\sqrt{\beta^2+(1+\beta)\frac{\mu}{\gamma}}-\beta}{2} = \frac{1}{2} \sqrt{\frac{\mu}{\gamma}} \left(\sqrt{\frac{1}{25}+1+\frac{4}{25} \sqrt{\frac{\mu}{\gamma}}}-\frac{1}{5} \right) \geq \frac{2}{5} \sqrt{\frac{\mu}{\gamma}}.
\end{equation*}

By the definition of $\tau_k$ we get the desired result.
\end{proof}

We also rephrase the previous result in terms of iteration complexity:

\begin{theorem}
Algorithm \ref{alg:second_version} starting from $\mathcal{X}_0$ satisfying
     \begin{equation*}
         \textnormal{dist}(\mathcal{X}_0,\mathcal{V}_{\alpha}) \leq \frac{1}{8} \sqrt{c_Q} \left(\frac{\delta}{\gamma} \right)^{3/4},
     \end{equation*}  
     computes an estimation $\mathcal{X}_T$ of $\mathcal{V}_{\alpha}$ such that $\textnormal{dist}(\mathcal{X}_T,\mathcal{V}_{\alpha}) \leq \epsilon$ in at most
\begin{equation*}
    T= \bigO \left(\sqrt{\frac{\gamma}{\delta }} \log \frac{f(\mathcal{X}_0)-f^*}{ \varepsilon \delta } \right)
\end{equation*}
many iterates.
\end{theorem}

\begin{proof}
For $\textnormal{dist}(\mathcal{X}_T,\mathcal{V}_{\alpha}) \leq \epsilon$, it suffices to have
\begin{equation*}
    f(\mathcal{X}_T)-f^* \leq c_Q \epsilon^2 \delta
\end{equation*}
by quadratic growth of $f$ (Proposition \ref{prop:quadratic_growth}). Using $(1-c)^k \leq \exp(-c k)$ for all $k \geq 0$ and $0 \leq c \leq 1$, Propositions \ref{prop:first_conv} and \ref{prop:main_conv} give that it suffices to choose $k$ as the smallest integer such that 
\begin{equation*}
    f(\mathcal{X}_k)-f^* \leq \exp\left(- \frac{2}{5} \sqrt{\frac{\mu}{ \gamma}} k \right) (f(\mathcal{X}_0)-f^*) \leq c_Q \epsilon^2 \delta.
\end{equation*}
Solving for $k$ and substituting $\mu = 2 c_Q \delta$, we get the required statement.
\end{proof}

\paragraph{Remark} Since the expression
\begin{equation*}
    \frac{\sqrt{\beta^2+(1+\beta)\frac{\mu}{\gamma}}-\beta}{\sqrt{\beta^2+(1+\beta)\frac{\mu}{\gamma}}+\beta} \cdot \mu
\end{equation*}
is strictly increasing with respect to $\mu$, we can choose $\gamma_0$ by substituting $\mu$ with an over-approximation, for example $\gamma$:
\begin{equation*}
    \gamma_0 = \frac{\sqrt{\beta^2+\beta+1}-\beta}{\sqrt{\beta^2+\beta+1}+\beta} \cdot \gamma
\end{equation*}

\section{Implementation details and computational cost of Algorithm \ref{alg:second_version}}

A naive implementation of Step 4 of Algorithm \ref{alg:second_version} can become quite costly as a simple binary search may need many function evaluations to reach $\beta_k$ in a good accuracy and as a result many large matrix-vector multiplications. Fortunately, we can manipulate the expressions so that it suffices to do only two large matrix-vector multiplications. The idea for such a technique comes from \cite{alimisis2023gradient}.

Let $\mathcal{X}$ be a point on Grassmann and $P$ a search direction. We consider the function
\begin{align*}
    \mathcal{X}(\eta) = \Exp_{\mathcal{X}} (\eta P) = \myspan(X V \cos(\eta \Sigma) V^T + U \sin(\eta \Sigma) V^T),
\end{align*}
where $X$ is a representative of $\mathcal{X}$ and $U \Sigma V^T$ a compact SVD of $P$. Here $\Sigma$ is taken as a diagonal matrix and the functions $\sin$ and $\cos$ act only to its diagonal entries.

The value of $f$ evaluated at $\mathcal{X}(\eta)$ is
\begin{align*}
    f(\mathcal{X}(\eta)) = & -\Tr((X V \cos(\eta \Sigma) V^T + U \sin(\eta \Sigma) V^T)^T A (X V \cos(\eta \Sigma) V^T + U \sin(\eta \Sigma) V^T)) \\ = & -\Tr(V \cos(\eta \Sigma) V^T X^T A X V \cos(\eta \Sigma) V^T ) - \Tr(V \sin(\eta \Sigma) U^T A U \sin(\eta \Sigma) V^T) \\ & - \Tr(V \cos(\eta \Sigma) V^T X^T A U \sin(\eta \Sigma) V^T) - \Tr(V \sin(\eta \Sigma) U^T A X V \cos(\eta \Sigma) V^T)   \\ =& -\Sigma_{i=1}^p (\cos^2(\eta \Sigma_i) \alpha_i + 2 \sin(\eta \Sigma_i) \cos(\eta \Sigma_i) \beta_i+\sin^2(\eta \Sigma_i) \gamma_i),
\end{align*}
where
\begin{equation*}
    \alpha_i=(V^T X^T A X V)_{ii},\hspace{1mm} \beta_i = (V^T X^T A U)_{ii} \hspace{2mm} \text{and} \hspace{2mm} \gamma_i = (U^T A U)_{ii}.
\end{equation*}

Thus, for computing the steps 4-5 of Algorithm \ref{alg:second_version}, we need to compute the matrix-vector products $A V_k$ and $A U$, where $U$ is the first matrix in the SVD of $\Log_{\mathcal{V}_k}(\mathcal{X}_k)$. Then, we can execute binary search (or any accelerated version, including Newton's method) for calculating $\beta_k$ without needing to compute any additional matrix-vector products with $A$. Moreover, these calculations are enough to provide immediately the product $A Y_k$ as
\begin{equation*}
    A Y_k = (A V_k) V \cos(\beta_k \Sigma) V^T + (A U) \sin(\beta_k \Sigma) V^T,
\end{equation*}
where $U \Sigma V^T$ is the SVD of $\Log_{\mathcal{V}_k}(\mathcal{X}_k)$. Thus, for computing the gradient step (step 7) in Algorithm \ref{alg:second_version}, we do not need to compute any new matrix-vector products as  $A Y_k$ suffices for calculating $\textnormal{grad}f(\mathcal{Y}_k)$. Consequently, the cost of computing one iteration of Algorithm \ref{alg:second_version} is two matrix-vector products. This is more than steepest descent or conjugate gradients method \cite{alimisis2023gradient} (that need only one matrix-vector product), but still reasonable as accelerated gradient descent typically features three kind of iterates ($\mathcal{X}_k$, $\mathcal{Y}_k$, $\mathcal{V}_k$).

\section{Numerical experiments}

\begin{table}[t]
    \centering
    \begin{tabular}{c c  c c c}
         \hline
         Matrix & $n$  & $\kappa$ & structure & fraction of nnz\\
         \hline 
         \texttt{FD3D} & $35\,000$ & $7.0\cdot 10^3$ & real & $1.95\cdot 10^{-4}$\\
         \texttt{ukerbe1} & $5\,981$ &  $-$ & rank-deficient, binary& $4.39 \cdot 10^{-4}$\\
         \texttt{ACTIVSg70K} & $69\,999$ & $2.9\cdot 10^8$ & real & $4.87\cdot 10^{-5}$ \\
         \texttt{boneS01} & $127\,224$ &  $4.2\cdot 10^{7}$ & real  & $3.40 \cdot 10^{-4}$ \\
         \texttt{audikw\_1} & $943\,695$ &  $-$ & rank-deficient, real & $8.7 \cdot 10^{-5}$ \\
         \hline
    \end{tabular}
    \caption{Summary statistics of the tested matrices.}
    \label{tab:matrix_summary}
\end{table}


We test the proposed method on a series of benchmark test matrices from the SuiteSparse Matrix Collection \cite{davis2011sparse} used also by \cite{saadNumericalMethodsLarge2011}. The main properties of the tested matrices, including their size $n$, condition number $\kappa = \lambda_1/\lambda_n$ and other structural properties, are summarized in Table~\ref{tab:matrix_summary}. For each of the tested problems we also report  the condition number
\begin{equation}
    \kappa_\mathrm{R} = \frac{\lambda_1 - \lambda_n}{\lambda_p - \lambda_{p+1}} = \bigO\left( 1/\delta \right),
\end{equation}
of the Riemannian Hessian evaluated the dominant subspace with spectral gap $\delta$ \cite{alimisis2022geodesic, alimisis2023gradient}.

\begin{table}[t]
    \centering
    \begin{tabular}{c | c c c c}
         Method & Mat.~products per iter.\\
         \hline
         Riem.~steepest descent  & 1 \\
         Chebyshev filter & 1\\
         BlockRQ RCG & 1 \\
         Nesterov acceleration & 2 
    \end{tabular}
    \caption{Comparison of the number of matrix products with $A$ required by each of the methods. \emph{Riem.\ steepest descent and Chebyshev filter subspace method require an additional matrix product every $k$ iterations where $k$ corresponds to the degree of the filter polynomial}.\label{tab:cost_per_iter}}
\end{table}

We compare the efficiency of the proposed Nesterov acceleration with three other methods: Riemannian steepest descent, Chebyshev filter in subspace iteration, and the Riemannian conjugate gradient for block Rayleigh quotient (BlockRQ RCG). We precompute the eigenvalues using \texttt{eigs} command in \uppercase{Matlab} required to determine the optimal Chebyshev filter for the subspace iteration and the parameters in the Nesterov acceleration and BlockRQ RCG. The tested algorithms differ in the number of matrix-vector products that they require per iteration, which we summarize in Table~\ref{tab:cost_per_iter}.


\paragraph{FD3D} We generate a 3D finite difference Laplacian matrix corresponding to the uniform grid of size $35\times 40 \times 25$ and zero Dirichlet boundary conditions, resulting in a matrix of size $35\,000$.

The four problems with the finite difference matrix \texttt{FD3D} are summarized in Table~\ref{tab:problems_fd3d}. We experiment with computing the dominant subspace of dimension $p=128$ and also with the minimization of the Rayleigh quotient.


\begin{table}[H]
    \centering
    \begin{tabular}{c | c c c c c}
         problem nb (\texttt{FD3D}) & type & $p$  & $\delta_p$ & $\kappa_\mathrm{R}$  & Cheb.\ degree \\
         \hline 
         1 & min & $128$ & $8.3 \cdot 10^{-4}$  &$1.4\cdot 10^4$ & $100$\\
         2 & max & $128$ & $8.3 \cdot 10^{-4}$  &$1.4\cdot 10^4$  & $100$
    \end{tabular}
    \caption{Tested problems for the \texttt{FD3D} matrix ($n = 35\,000, \kappa = 7.0\cdot 10^3$).}
    \label{tab:problems_fd3d}
\end{table}

\begin{figure}[t]
    \centering
    \begin{subfigure}[b]{0.45\textwidth}
         \centering
         \includegraphics[width=\textwidth]{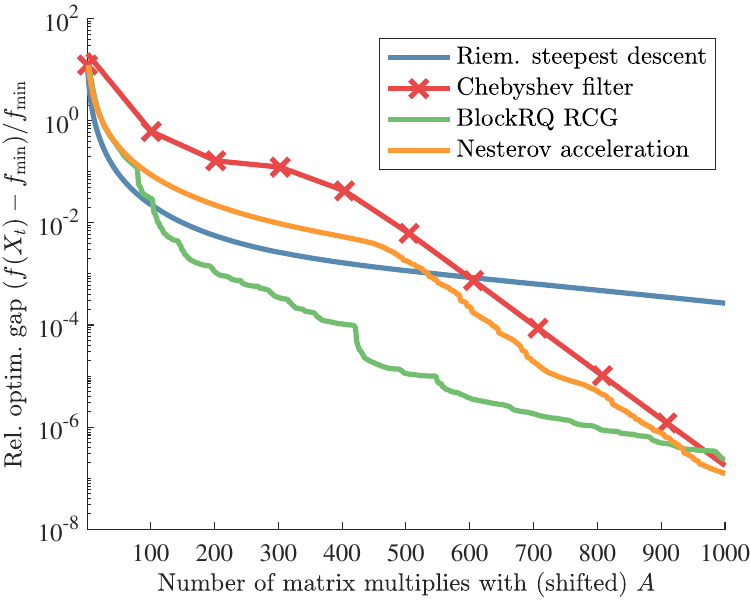}
         \caption{Problem nb 1 ($p=128$, min.)\label{fig:fd3d-p128-min_matvecs}}
    \end{subfigure}
    \begin{subfigure}[b]{0.45\textwidth}
         \centering
         \includegraphics[width=\textwidth]{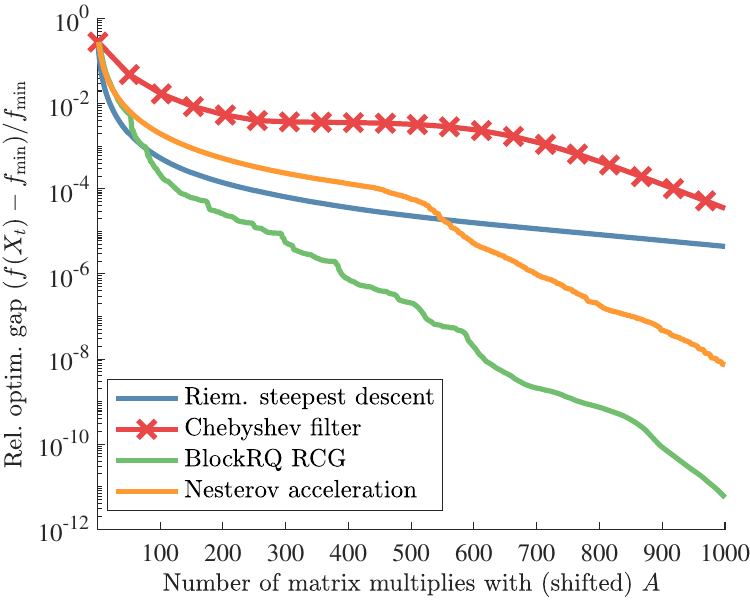}
         \caption{Problem nb 2  ($p=128$, max.)\label{fig:fd3d-128_matvecs}}
    \end{subfigure}
    \caption{The \texttt{FD3D} matrix.\label{fig:fd3d}}
\end{figure}

In Figure~\ref{fig:fd3d} we show the convergence plots tracking the number of matrix-vector products for problem nb 1 and 2. Overall, for both problems BlockRQ RCG outperforms the other methods, while Nesterov Acceleration matches the slow convergence of Riem.\ steepest descent for the first $250$ iterations after which it starts converging at the proved rate $\bigO(1/\sqrt{\delta})$. The Chebyshev subspace iteration with polynomial of degree $100$ is able to match the convergence rate of the Nesterov acceleration on problem nb 1 but not on problem nb 2.


\paragraph{ukerbe1} The matrix comes from a locally refined non-uniform grid of a 2D finite element problem. Although the matrix is of a smaller size $n = 5\,981$ compared to the other tested matrices, it is a more challenging problem due to the non-uniform grid resulting in a very high condition number. The two tested problems are summarized in Table~\ref{tab:problems_ukerbe1} and differ in the size of the subspace $p=64$ and $p=128$.

\begin{table}[H]
    \centering
    \begin{tabular}{c | c c c c c}
         problem nb (\texttt{ukerbe1}) & type & $p$  & $\delta_p$ & $\kappa_\mathrm{R}$ & Cheb.\ degree\\
         \hline 
         3 & max & $64$ & $1.2\cdot 10^{-3}$ & $5.2\cdot 10^3$ & $100$\\
         4 & max & $128$ & $9.4\cdot 10^{-4}$ & $6.7\cdot 10^3$ & $100$
    \end{tabular}
    \caption{Tested problems for \texttt{ukerbe1} rank-deficient matrix ($n =5\,981$).\label{tab:problems_ukerbe1}}
\end{table}

\begin{figure}[t!]
    \centering
    \begin{subfigure}[b]{0.45\textwidth}
        \centering
        \includegraphics[width=\textwidth]{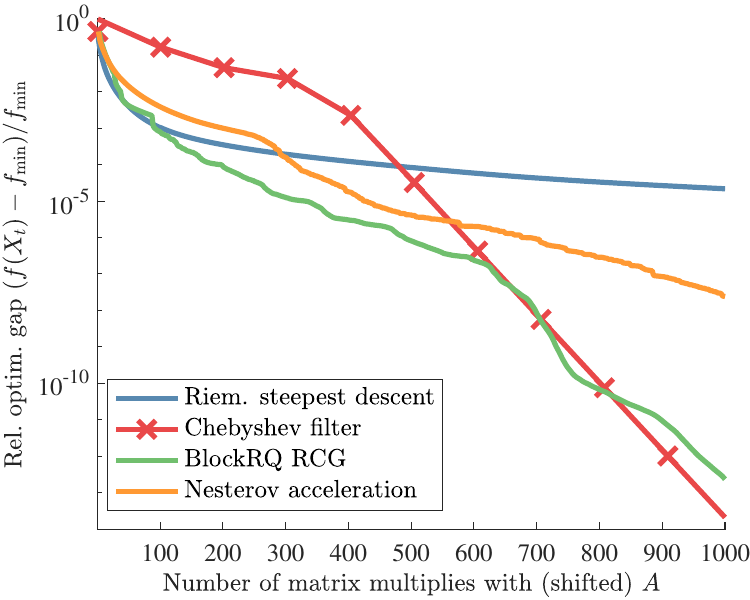}
         \caption{Problem nb 3 ($p=64$)\label{fig:ukerbe1-p64-max_iter}}
     \end{subfigure}
     \begin{subfigure}[b]{0.45\textwidth}
         \centering
         \includegraphics[width=\textwidth]{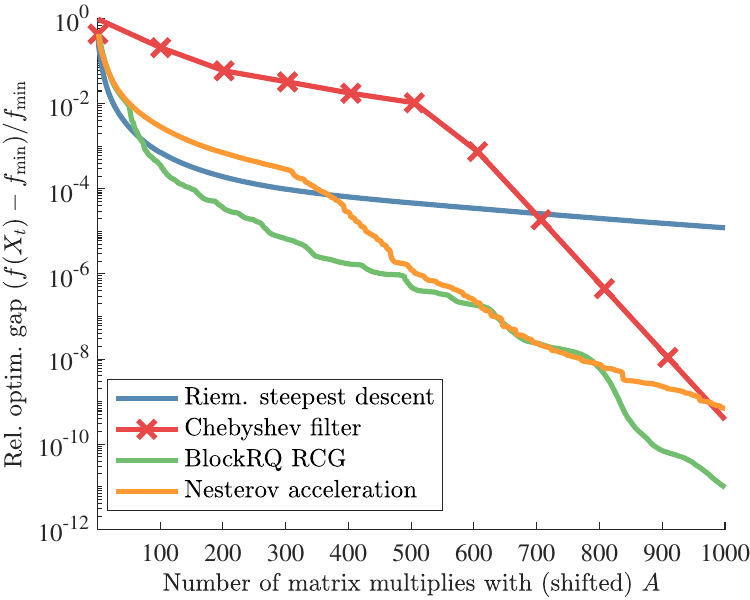}
         \caption{Problem nb 4 ($p=128$)\label{fig:ukerbe1-p128-max_matvecs}}
    \end{subfigure}
    \caption{Comparison on \texttt{ukerbe1} test matrix on problem nb 3 and 4. \label{fig:ukerbe1}}
\end{figure}
In Figure~\ref{fig:ukerbe1} we see the performance of the methods on problem nb 3 and 4 for \texttt{ukerbe1} matrix. In both problems BlockRQ RCG and Chebyshev subspace iteration eventually outperform the Nesterov acceleration. We also observe an initial slow convergence of the Chebyshev iteration until the iteration $400$ and $500$ respectively. 


\paragraph{ACTIVSg70K} This large matrix models a synthetically generated power system grid. We experiment with subspace dimension $p=32$ and $p=64$ as described in Table~\ref{tab:problems_ACTIVSg70K}.


\begin{table}[H]
    \centering
    \begin{tabular}{c | c c c c c}
         problem nb (\texttt{ACTIVSg70K}) & type & $p$  & $\delta_p$ & $\kappa_\mathrm{R}$ & Cheb.\ degree \\
         \hline 
         5 & max & $32$ & $2.2\cdot 10^2$  & $1.2\cdot 10^3$ & $50$ \\
         6 &  max & $64$ & $2.0\cdot 10^2$  & $1.3\cdot 10^3$ & $50$
    \end{tabular}
    \caption{Tested problems for \texttt{ACTIVSg70K} matrix ($n = 69\,999, \kappa = 2.9\cdot 10^8$).}
    \label{tab:problems_ACTIVSg70K}
\end{table}

\begin{figure}[t]
    \centering
    \begin{subfigure}[b]{0.45\textwidth}
        \centering
        \includegraphics[width=\textwidth]{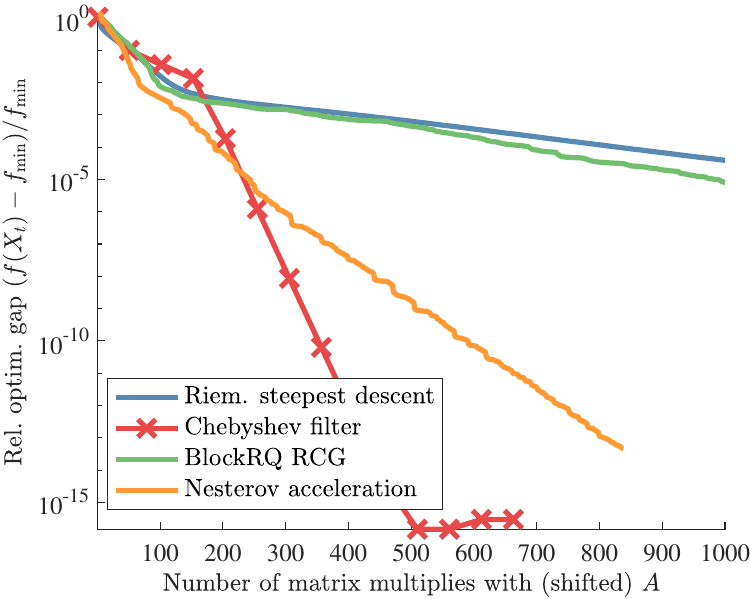}
         \caption{Problem nb 5 (max, $p=32$)}
     \end{subfigure}
     \begin{subfigure}[b]{0.45\textwidth}
         \centering
         \includegraphics[width=\textwidth]{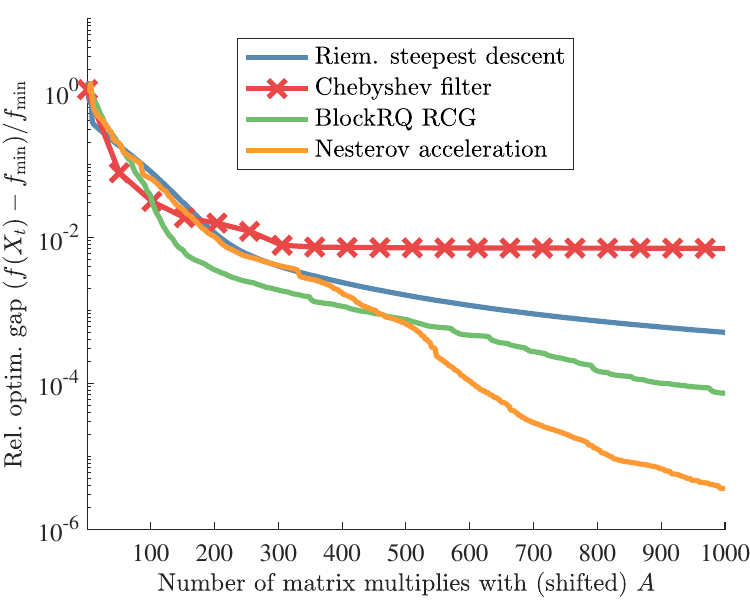}
        \caption{Problem nb 6 (max, $p=64$)}
    \end{subfigure}
    \caption{Comparison on \texttt{ACTIVSg70K} test matrix on problem nb 5 and 6.\label{fig:ACTIVSg70K}}
\end{figure}

Figure \ref{fig:ACTIVSg70K} shows the convergence plots for \texttt{ACTIVSg70K}. We see that Nesterov Acceleration outperforms the other methods on the problem with larger subspace dimension of $p=64$ (which is harder). For the problem with smaller subspace of dimension $p=32$, the Chebyshev iteration algorithm outperforms the other methods (while being the worst performing on $p=64$), which reveals its sensitivity to choosing the correct degree of the filter polynomial.


\paragraph{boneS01} The second largest matrix we test is of size $n=127\,224$ and comes from a finite element model studying the porous bone micro-architecture. The problem is challenging due to its large size and Riemannian condition number, see Table~\ref{tab:problems_boneS01}

\begin{table}[H]
    \centering
    \begin{tabular}{c | c c c c c}
         problem nb (\texttt{boneS01}) & type & $p$  & $\delta_p$ 
         & $\kappa_\mathrm{R}$ & Cheb.\ degree \\
         \hline 
         7 & max & $64$ & $2.4\cdot 10^1$ & $2.1\cdot 10^3$ & $50$ \\
         8 & max & $128$ & $1.3\cdot 10^1$ & $3.6\cdot 10^3$ & $50$
    \end{tabular}
    \caption{Tested problems for \texttt{boneS01} matrix ($n = 127\,224, \kappa = 4.2\cdot 10^7$).}
    \label{tab:problems_boneS01}
\end{table}

\begin{figure}[t]
    \centering
    \begin{subfigure}[b]{0.45\textwidth}
        \centering
        \includegraphics[width=\textwidth]{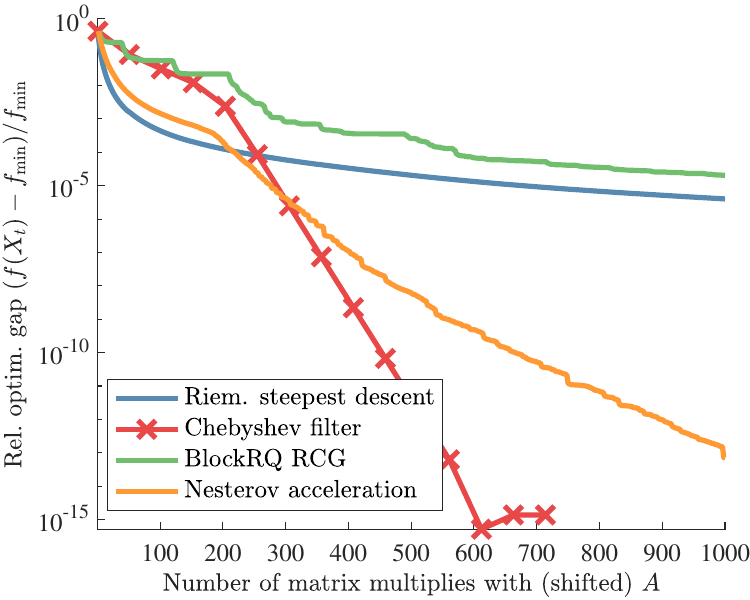}
         \caption{Prob nb 7 (max $p=64$)\label{fig:boneS01-p64-matvecs}}
     \end{subfigure}
     \begin{subfigure}[b]{0.45\textwidth}
         \centering
         \includegraphics[width=\textwidth]{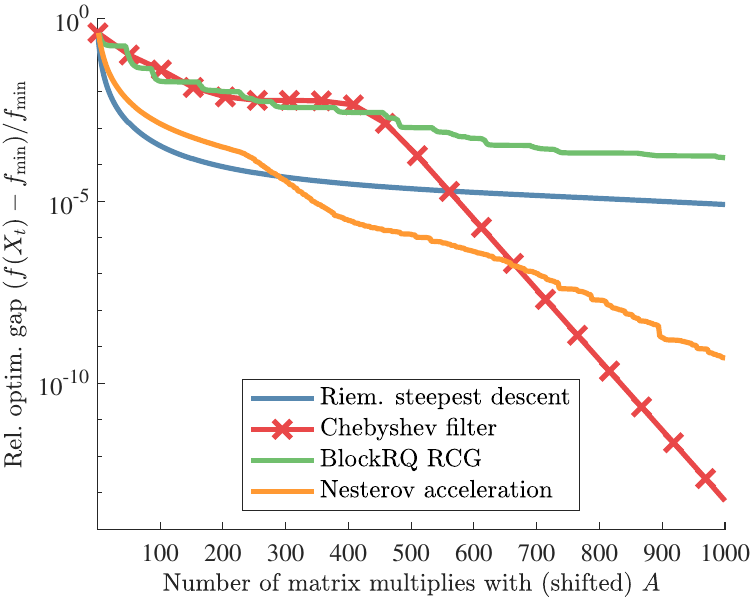}
         \caption{Prob nb 8 (max $p=128$)\label{fig:boneS01-p128-max_matvecs}}
    \end{subfigure}
    \caption{Comparison on \texttt{boneS01} test matrix on problem 7 and 8.\label{fig:boneS01}}
\end{figure}

Figure \ref{fig:boneS01} shows the performance of the methods on \texttt{boneS01}. We see that the Chebyshev subspace iteration with degree $50$, while having slower convergence at the beginning, outperforms the other methods. Nesterov acceleration is the second best performing and faster than Block RQ RCG and Riemannian steepest descent. 


\paragraph{audikw\_1}
The largest matrix of size $n=943\, 695$ we experiment with comes from finite elements problem modelling automotive crankshaft structure. The problem is challenging due to its large size and its large Riemannian condition number as can be seen in Table~\ref{tab:problems_audikw1}.
\begin{table}[h]
    \centering
    \begin{tabular}{c | c c c c c}
         problem nb (\texttt{audikw\_1}) & type & $p$  & $\delta_p$ & $\kappa_\mathrm{R}$ & Cheb.\ degree \\
         \hline 
         9 & max & $32$ & $4.3\cdot 10^6$ & $5.6\cdot 10^3$ & $25$\\
         10 & max & $64$ & $1.9\cdot 10^7$ & $1.3\cdot 10^3$ & $25$
    \end{tabular}
    \caption{Tested problems for \texttt{audikw\_1}  rank-deficient matrix ($n = 943\,695$).}
    \label{tab:problems_audikw1}
\end{table}

\begin{figure}[t]
    \centering
    \begin{subfigure}[b]{0.45\textwidth}
        \centering
        \includegraphics[width=\textwidth]{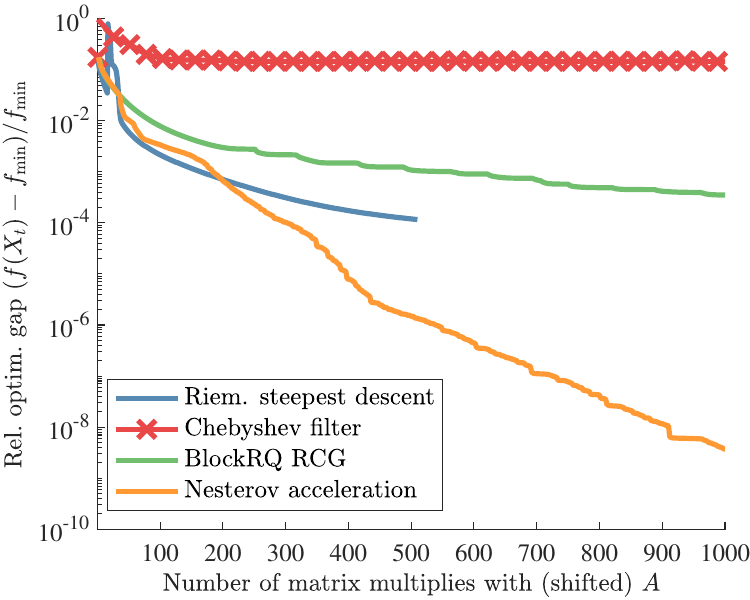}
        \caption{Prob nb 9 (max, $p=64$)\label{fig:audikw_1-p64-matvecs}}   
    \end{subfigure}
     \begin{subfigure}[b]{0.45\textwidth}
         \centering
         \includegraphics[width=\textwidth]{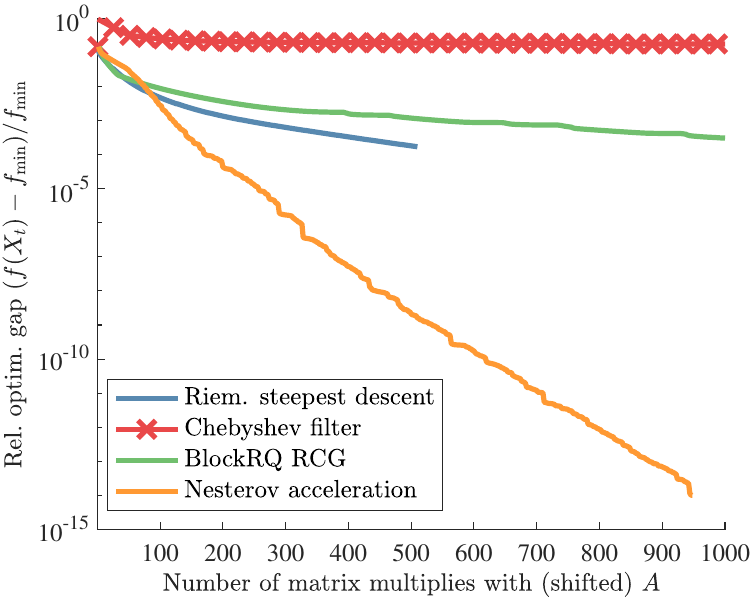}
         \caption{Prob nb 10 (max, $p=128$)\label{fig:audikw_1-p128-max_matvecs}}
    \end{subfigure}
    \caption{Comparison on \texttt{audikw\_1} test matrix on problem nb 9 and 10.\label{fig:audikw_1}}
\end{figure}

In Figure \ref{fig:audikw_1} we see the convergence results for the largest tested matrix \texttt{audikw\_1}. In this test, Nesterov acceleration clearly outperforms the other methods. The Chebyshev subspace iteration algorithm does not converge, which might be due to the wrong choice for the degree of the polynomial.

\section{Conclusion}

In this work, we brought together one of the most important problems in applied mathematics, namely the symmetric eigenvalue problem, and one of the most famous optimization algorithms, namely Nesterov's accelerated gradient descent. The result is an accelerated gradient descent algorithm on the Grassmann manifold that solves the symmetric eigenvalue problem with accelerated convergence guarantees which match the state-of-the-art. Numerically, our algorithm shows promising performance and in some large problems outperforms other important eigenvalue solvers. There are two main disadvantages that can be fruitful ground for future work. First, the convergence analysis is very complicated and the resulting convergence rate only local. Second, the performance of the algorithm depends heavily on the choice of hyperparameters, which are not easy to be known in advance. One can try to work with more advanced versions of Riemannian gradient descent than the one presented in \cite{zhang2018towards}, which hopefully will simplify the analysis, improve the convergence guarantees and free the algorithm from a heavy selection of hyperparameters. In any case, we expect the quasi-convexity structure of eigenvalue problems proved in \cite{alimisis2022geodesic} to be the cornerstone of all such efforts.

\bibliographystyle{plain} 
\bibliography{refs}



\newpage

\vskip 0.2in

\end{document}